\documentclass{article}

\usepackage{amssymb,bbm,graphicx,epsfig,psfrag,epic,eepic,latexsym, hyperref}
\usepackage{amsthm}
\usepackage{amsmath}
\usepackage{pgfplots}
\pgfplotsset{compat=newest}
\usetikzlibrary{intersections, pgfplots.fillbetween}
\usetikzlibrary{cd,patterns, calc}
\usepackage{comment}

\theoremstyle{definition}

\newtheorem{thm}{Theorem}[section]
\newtheorem{prop}{Proposition}[section]
\newtheorem*{prop*}{Proposition}
\newtheorem{defn}{Definition}[section]

\newtheorem{rem}{Remark}[section]
\newtheorem{lemma}{Lemma}[section]
\newtheorem{ex}{Example}[section]

\usepackage{mathrsfs}
\usepackage{color}

\newcommand{\N}{\mathbb{N}}

\newcommand{\Z}{\mathbb{Z}}
\newcommand{\R}{\mathbb{R}}
\newcommand{\C}{\mathbb{C}}

\newcommand{\crit}{\mathrm{Crit}}

\newcommand{\Address}{{
  \bigskip
  \footnotesize
  \textsc{Lehrstuhl für Analysis und Geometrie, Universität Augsburg, D-86135 Augsburg, Germany}\\
  \textit{E-mail address:} \href{mailto:filipvbrocic@gmail.com}{\tt filipvbrocic@gmail.com}
}}

\begin{document}
\title{Wrapped Floer homology and subcritical handle attachment}
\author{Filip Bro\'ci\'c}
\maketitle

\begin{abstract}
In this expository article, we present the proof of the invariance of the wrapped Floer homology under the subcritical handle attachment. This is proved in \cite{Ir13}. Here, we fix a minor gap in the proof about the choice of a cofinal family of Hamiltonians. We adapt the arguments from \cite{Fa16-phd, Fa20}, where the gap was resolved for the case of handle attachment in symplectic homology. The effect of the handle attachment on the symplectic homology was originally explored in \cite{Ci02}.  
\end{abstract}

\section{Introduction}\label{sec:intro}

The purpose of this expository article is to prove the invariance theorem for subcritical handle attachment, while providing a fairly detailed introduction to the objects involved in the statement. Our goal was to present a mostly self-contained proof; however, it would be a bold claim to say that we have fully achieved this.

One way to describe wrapped Floer homology is as a Lagrangian analogue of symplectic homology. To elaborate a bit further, Floer introduced an infinite-dimensional Morse theory to study periodic orbits of Hamiltonian systems on closed symplectic manifolds. Symplectic homology is a variant of Floer theory that helps us study periodic orbits of Hamiltonian systems that are controlled at infinity on certain non-compact symplectic manifolds. Wrapped Floer homology studies Hamiltonian chords with endpoints on Lagrangian submanifolds belonging to an appropriate class of non-compact Lagrangians, in an analogous setup to symplectic homology. For definitions of all the objects involved, see Section \ref{sec:geometry_background}.

We have chosen to include a complete statement of the main result in the introduction to facilitate citation. Readers who are not familiar with the objects appearing in the statement will find them gradually introduced in Sections \ref{sec:geometry_background}, \ref{sec:wrapped_floer}, and \ref{sec:handle_attach}.

\begin{thm}\label{thm:main}
Let $(M, \lambda)$ be a Liouville domain such that $2c_1(M) = 0$. Let $L_0, L_1$ be two exact cylindrical Lagrangians which satisfy that the Maslov classes $\mu^{\Theta}_{L_i} \in H^1(L_i)$ vanish, with respect to the same non-vanishing section $\Theta$ of the square of the complex determinant line bundle $(\bigwedge^n_{\C} TM)^{\otimes 2}$. If $S \subset L_1 \cap \partial M$ is an isotropic sphere with a trivial conformal normal symplectic bundle, then
$$
HW_*(L_0, L_1 ; M) \cong HW_*(L_0, L_1 \cup_{S} H_k^{n} ; M \cup_{S} H_k^{2n}).
$$
Here $H_k^{2n}$ is the standard Weinstein $k$-handle, and $H_k^n$ is its imaginary part. If $L:=L_0 = L_1$,  then it is enough to assume that $2c_1(M,L) = 0 \in H^2(M,L)$. In this case the statement is $HW_*(L;M) = HW_*(L\cup_{S} H_k^{n}; M \cup_{S} H_k^{2n})$.
\end{thm}
\begin{rem}
	It follows from the proof that one can also allow Lagrangians $L_1$ that are linear in the handle region, invariant under the Liouville flow in $H_{k}^{2n}$, and such that the function $\mathcal{L}= \sum_{i=1}^{k} x_i y_i$ vanishes along them. The Lagrangians in $H_k^{2n}$ that satisfy both of these conditions are linear Lagrangians of the form $\nu V_0 \times U_1 \subset \C^k \times \C^{n-k}$, where $V_0 \subset \R^k$, $\nu V_0 \subset \C^k$ is its conormal, and $U_1 \subset \C^{n-k}$ is any linear Lagrangian through the origin. In particular, one can take $L_0 = L_1 = \emptyset$, and take the horizontal Lagrangian $i H_k^n$. Then, the invariance result says that after the Lagrangian $iH_k^n$ has vanishing wrapped Floer homology, which is expected since $HW_*(\R^n; \C^n) = 0$. 
\end{rem}

\subsection{Applications}

In this section, we will sketch the applications of the handle attaching theorem from \cite{Ci02, Se07, Mc09, Ir13,BF25}. Even though we do not cover symplectic homology in the present article, we start with the original application of the invariance theorem to Arnol`d's chord conjecture from \cite{Ci02}. 

The Arnol`d's chord conjecture as stated in \cite{Ar86} guarantees that for any contact form $\alpha$ that induces the standard contact structure $(S^3, \xi_{st})$, and for any Legendrian $\Lambda \subset S^3$ there exists a Reeb chord with endpoints on $\Lambda$. In \cite{Ci02} the following theorem was proved.

\begin{thm}
	The chord conjecture holds for a standard Legendrian unknot sphere $\Lambda$ in $(Y, \xi)$ where $Y$ is the boundary of a subcritical Weinstein domain $W$.
\end{thm}
A Weinstein domain is subcritical if it is obtained by attaching subcritical handles to the standard ball $B^{2n}$. This means that the isotropic spheres along which the handles are attached are of dimension $k-1 < n-1$. The proof consists of two steps. Firstly, one can show that if we attach a critical handle $H_n^{2n}$ to $W$ along $\Lambda$ then the result of the handle attaching is Weinstein homotopic to $T^*S^{n}$ with subcritical handles attached \cite[Proposition 2.9]{Ci02}. Hence, we know that using the subcritical handle attachment, together with an isomorphism from symplectic homology of $T^*S^n$ with the homology of the free loop space $\Lambda S^n$ we have $SH_*(W \cup_{\Lambda} H_n^{2n}) \cong H_*(\Lambda S^n)$. On the other hand, by \cite[Theorem 1.11]{Ci02} we know that the dimension of the symplectic homology of $SH_*(W \cup_{\Lambda} H_n^{2n})$ differs by at most one from the symplectic homology $SH_*(W)$ if $\Lambda$ poses no Reeb chords\footnote{The analogue of such a statement we do not cover in the present article.}. Now, since $W$ is subcritical, it has the same symplectic homology as $B^{2n}$, hence it vanishes. This leads to a contradiction since $H_*(\Lambda S^n)$ is infinite dimensional, hence $\Lambda$ must have chords. 

Another application is the proof that every Liouville domain $W$ whose boundary is contactomorphic to the standard contact sphere $(S^{2n-1}, \xi_{st})$ has vanishing symplectic homology. Such a domain is called a \textit{Liouville filling}. This appears in \cite[Corollary 6.5]{Se07} and is attributed to Ivan Smith. The proof is using three main ingredients. The first one is a result by Eliashberg, Floer, and McDuff (see \cite[Theorem 5.1]{El91}) that every Liouville filling $W$ of the standard contact sphere $(S^{2n-1}, \xi_{st})$ is diffeomorphic to the ball $B^{2n}$. The second one is the spectral sequence that converges to $SH_*(W)$ for Liouville domains whose boundary $\partial W$ admits a contact form with periodic Reeb flow (\cite[Equation 3.2]{Se07}), and lastly, the invariance of symplectic homology under contact connected sum which is a special case of the subcritical handle attachment. 

Note that $W$ being diffeomorphic to $B^{2n}$ does not imply that $SH_*(W) = 0$. There are examples of \textit{exotic} symplectic structures on $\R^{2n}$, that are completions of a Weinstein domain. See \cite{SS05} for the case $\R^{4m}$ with $m>1$, and \cite{Mc09} for $\R^{2n}$, $n>3$. Due to a theorem of Gromov \cite{Gr85} it is known that every Liouville filling of the standard contact 3-sphere $(S^3, \xi_{st})$ is Liouville isomorphic to $\R^{4}$. In fact, in \cite{Mc09} used the invariance of symplectic homology under contact connected sum to build infinitely many exotic Stein structures on $\R^{2n}$. He distinguished them by the number of idempotent elements, where the product structure on $SH_*(W)$ is defined using the \textit{pair of pants} configurations.

For more details on symplectic (co)homology, we refer to survey articles \cite{Oa04, Se07, We}. For a complete proof of the relationship between the symplectic (co)homology of cotangent bundles and the homology of the free loop space of the base, we refer to \cite{Ab15}. The isomorphism over $\Z_2$ was established in \cite{Vi99}. The isomorphism with $\Z$ coefficients was established in \cite{SW06, ASch06}, and that the pair of pants product on symplectic (co)homology corresponds to the Chas-Sullivan product was proved in \cite{ASch10}.

Arnol`d's chord conjecture, in its original form, was completely solved in \cite{Mo01}, where it was proved that for a boundary of a subcritical Stein domain and any Legendrian $\Lambda$ there exists a Reeb chord. The natural generalization of the chord conjecture was solved for all closed $3$ contact manifolds and all closed Legendrian curves in \cite{HT11, HT13}. In \cite{BCS24} it was shown that the chord conjecture holds for Legendrians that are isotopic to the conormal lift in $T^*N$, with the standard contact structure on $S^*N$.

In \cite{Ir13} used the invariance theorem to prove the existence of non-trivial Reeb chords with endpoints on the zero section $N \subset T^*M$ on the energy hypersurface $H_V^{-1}(0)$ where:
$$
H_V(q,p) = \frac{1}{2} \| p \|^2 + V(q).
$$
Note that such chords obviously do not exist if $H_V^{-1}(0) \cap N = \emptyset$. In this case, $H_V^{-1}(c)$ is star-shaped, hence it is diffeomorphic to $S^* N$. Assume that $V$ is Morse, and that the index $n$ critical points of $V$ have critical value bigger than $0$. In the case $H_V^{-1}(0) \cap N \neq \emptyset$, $H_V^{-1}(0)$ is obtained by a sequence of subcritical handle attachments to $(B^{2n}, \R^n)$, hence by the invariance theorem
$$
HW_*(N \cap \{H_V \leq 0\}; \{H_V \leq 0\}) \cong HW_*(\R^n, B^{2n}) \cong 0.
$$
Since the wrapped Floer complex is generated by critical points of a Morse function on the Lagrangian, together with Reeb chords, there must be a Reeb chord with endpoints on $N$ because wrapped Floer homology vanishes.

In \cite{BF25}, the invariance theorem is used to prove the existence of infinitely many consecutive collisions on the energy hypersurface slightly above the first critical value in the circular restricted three-body problem. The Hamiltonian for the massless \textit{satellite} is given on $T^* \R^{2}\setminus \{e,m\}$ by
$$
H(q_1,q_2,p_1,p_2) = \frac{1}{2}|p|^2 - \frac{\mu}{|q-m|} - \frac{1- \mu}{|q-e|} +q_1 p_2 - q_2 p_1,
$$
where $e,m \in \R^{2}$ are the \textit{Earth} and the \textit{Moon}. The energy hypersurface below the first critical value has two bounded components, one containing $m$, the other containing $e$. Both of them can be regularized as a boundary of a fiber-wise star-shaped domain in $T^*S^2$, i.e. they are diffeomorphic to $\R P^3$ with a standard contact structure (see \cite[Corollary 1.5]{AFKP12}). The energy hypersurface slightly above the first critical values is shown to be diffeomorphic to the contact connected sum $\R P^3 \# \R P^3$ (see \cite[Corollary 1.5]{AFKP12}), and the contact connected sum is an example of a sub-critical handle attachment, hence the invariance theorem can be applied. Since consecutive collisions correspond to the Reeb chords with endpoints on the fiber in $T^*S^2$ below the first critical value, by the invariance theorem, we know that the wrapped Floer homology remains unchanged. In this case, it is isomorphic to the singular homology $H_*(\Omega_q S^2; \Z_2)$ of the based loop space $\Omega_q S^2$, and we know that $H_*(\Omega_q S^2; \Z_2)$ is infinite-dimensional.

\subsection{Organization}

In Section \ref{sec:geometry_background} we define the geometric setup: Liouville domains, (exact, cylindrical) Lagrangians, Hamiltonian diffeomorphisms, and almost complex structures. 

Section \ref{sec:wrapped_floer} involves the definition of the wrapped Floer homology groups $HW_*(L_0, L_1)$. To avoid additional complications with orientations, our homology groups will be defined over $\Z_2$. The assumptions about the first Chern classes are related to the global $\Z$ grading of the chain complex. We will explain how to use a certain PDE to define a differential and sketch the analysis behind the nice properties of the space of solutions. The homology $HW_*(L_0, L_1)$ is defined by a direct limit of $HF(L_0, L_1; H_i)$ where $H_i$, $i\in \N$ are \textit{linear at infinity}; we need to explain how to relate $HF(L_0, L_1; H_i)$ and $HF(L_0, L_1; H_j)$, and when is this possible.

In Section \ref{sec:transf_morph}, we present Viterbo's transfer morphism in the case of wrapped Floer homology following \cite{Vi99, AS10, Fa20}.

In Section \ref{sec:handle_attach}, we define contact surgery along an isotropic sphere $\Lambda$ following \cite{We91}. Topologically, it is the same as the standard surgery; on top of that, the surgered manifold has a contact structure. The trace of the surgery is a symplectic cobordism. In particular, for contact manifolds $Y$ that are fillable by a Liouville domain $M$, the surgered contact manifold is filled by $M \cup_{\partial M} W$, where $W$ is the symplectic cobordism given by the trace of the surgery.

Section \ref{sec:inv_handle} contains the proof of Theorem \ref{thm:main}. We follow the ideas of \cite{Fa16, Fa16-phd, Fa20} on how to define the cofinal family $\{H_i\}_{i \in \N}$ with the desired properties. The essence of the proof is the index argument about the control of the newly created chords in $H_k^{2n}$. That is why we need to have well-defined integer grading on the wrapped Floer homology. The gap in \cite{Ci02, Ir13} is that the cofinal family with the listed properties (see \cite[Lemma 2.5]{Ci02} and \cite[p. 393-394]{Ir13}) can not exist, and when this is resolved one creates more orbits then just the critical point corresponding to the origin in $H_k^{2n}$. However, the index of these orbits can be controlled, and the proof remains valid.

\section*{Acknowledgments}
This expository article was written while I was giving a mini-course on the invariance theorem at the University of Augsburg. I am grateful to Urs Frauenfelder, who, despite his expertise in the subject, attended this mini-course. I am also thankful to Zhen Gao for pointing out typos and inconsistencies in terms of conventions and notations. Additionally, I would like to thank Kai Cieliebak for many useful discussions. This project was supported by the Deutsche Forschungsgemeinschaft (DFG, German Research Foundation) – 517480394. 
\section{Symplectic geometry}\label{sec:geometry_background}

We assume some basic knowledge of symplectic geometry. We cover the definitions of objects that are relevant to the present situation. For more details on symplectic geometry, one can consult a wonderful textbook \cite{MS17}.

\begin{defn}
A Liouville domain is an exact symplectic manifold $(M, d \lambda)$ such that the one form $\lambda$ restricted to the boundary $\partial M$ is a contact form that induces the same orientation on $\partial M$ seen as the boundary of $M$. The primitive $\lambda \in \Omega^1(W)$ is called the Liouville form. 
\end{defn}

Given a Liouville domain $(M, \lambda)$ there is a canonical vector field $X_{\lambda}$ uniquely determined by $i_{X_{\lambda}} d \lambda= \lambda$. Uniqueness follows from the non-degeneracy of the symplectic form $\omega = d \lambda$. The vector field $X:=X_{\lambda}$ is called the \textit{Liouville vector field}, and we denote the Liouville domain by a triple $(M, \lambda, X)$. The fact that the contact form $\alpha=\lambda \vert_{\partial M}$ induces the same orientation as the induced one on $\partial M$ is equivalent to saying that the Liouville vector field $X$ is positively transverse to the boundary. Transversality is equivalent to the contact condition, and positivity is related to the orientation.

\begin{rem}
Every compact exact symplectic manifold $(M, d \lambda)$ has a non-empty boundary $\partial M$. Indeed, by Stoke's theorem, we have if $\partial M = \emptyset$:
$$
0 < \int_M \omega^n = \int_{\partial M} \lambda \wedge \omega^{n-1} = 0,
$$
which is a contradiction.
\end{rem}

Now, we give two main examples of a Liouville domain.
\begin{ex}\label{ex:ball}
A subset  $X \subset V$ of a vector space $V$ is called star-shaped with respect to $p \in X$ if for every $q \in X$ we have that the segment $$[p,q]:= \{t p + (1-t) q \mid t \in [0,1]\},$$ is a subset of $X$. Let $\Omega \subset \R^{2n}$ be a compact domain, star-shaped with respect to the origin $0 \in \R^{2n}$  with a smooth boundary $\partial \Omega$, then $$\left(\Omega, \frac{1}{2} \sum_{i} x_i dy_i - y_i dx_i, \frac{1}{2} \sum_{i} x_i \partial_{x_i} + y_i \partial_{y_i}\right),$$ is a Liouville domain. In particular, the unit ball $B^{2n}$ is a Liouville domain, and its boundary is a contact sphere with the standard contact structure.
\end{ex}

\begin{ex}\label{ex:codisk}
Let $Q$ be a smooth closed manifold, and let $T^*Q$ be its cotangent bundle. Cotangent bundle $T^*Q$ has a canonical symplectic form given by the differential of the Liouville form $\lambda_{can} \in \Omega^1(T^*Q)$, where for $Y_p \in T_p T^*Q$ we set
$$
\lambda_{can}(p) (Y_p) = p (d \pi (Y_p)),
$$
where $\pi: T^*Q \to Q$ is the projection. Let $\Omega \subset T^*Q$ be a compact, fiber-wise star-shaped domain with a smooth boundary. Then $(\Omega, \lambda_{can}, X)$ is a Liouville domain where $X(p) = p$. This equality should be understood through the canonical identification between the fiber $T^*_{\pi(p)} Q$ and the fiber of the vertical sub-bundle $V_p = \{ Y \in  T_p T^*Q \mid d \pi (Y) = 0 \} $.  In local coordinates induced by a local chart $q_1, ..., q_n \in \mathcal{U} \subset Q$ one has $\lambda_{can} = \sum p_i d q_i$ and $X = \sum p_i \partial_{p_i}$. In particular, if one fixes a Riemannian metric $g$ on $Q$, we have that the unit codisk bundle
$$
D^*_g Q := \{p \in T^*Q \mid \|p\|_g^* \leq 1\}
$$
is a fiber-wise convex (hence star-shaped) domain with a smooth boundary.
\end{ex}
Equality $i_X d \lambda = \lambda$ implies $\lambda(X) = 0$. By Cartan's magic formula, we have that the flow $\varphi^t$ of $X$ satisfies:

$$
\frac{d}{dt} \varphi_t^* \lambda = \varphi_t^*(d(i_X \lambda) + i_X d\lambda) = \varphi_t^*\lambda,
$$
so we get $\varphi_t^* \lambda = e^t \lambda$. Since $X$ is positively transverse to $\partial M$, and $M$ is compact, we have that the flow is defined for $t \in (-\infty, 0]$. The set $$\mathrm{Core}(M, X) = \bigcap_{t\leq 0} \varphi_t(M)$$ is called the core of a Liouville domain.
\begin{lemma}\label{lemma:collar}
The map
$$
\begin{aligned}
\psi:  (\partial M \times (0,1], r \alpha) &\to M \setminus \mathrm{Core}(M, X)\\
(p, r) &\mapsto \varphi_{\log (r)}(p) 
\end{aligned}
$$
is a diffeomorphism satisfies $\psi^* \lambda = r \alpha$
\end{lemma}
\begin{proof}
It follows from the uniqueness of the solution of ODEs that $\psi$ is injective, and from the definition of $\mathrm{Core}(M, X)$ that $\psi$ is surjective. Since $\varphi_t^* \lambda = e^t \lambda$, and $\alpha = \lambda \vert_{\partial M}$ we have that for $(Y, s) \in T_p \partial M \times \R$
$$
\begin{aligned}
\psi^* \lambda (Y,s) &= \psi^* \lambda (Y, 0) + \psi^* \lambda (0,s) \\
&= \varphi^*_{\log r} \lambda (Y) + s \lambda \left( \frac{d}{dr} \varphi_{\log r} (p) \right)\\
&= e^{\log r} \lambda ( Y ) + s \lambda \left( \frac{1}{r} X(\varphi_{\log r} (p) )\right) = r \alpha (Y).
\end{aligned}
$$
\end{proof}

It follows from the previous lemma that we can \textit{complete} the Liouville domain by gluing the positive part $(\partial M \times [1, +\infty), r \alpha)$ of the symplectization of $(\partial M, \alpha)$. In the collar neighborhood, $X$ is identified with $\partial_r$, so this provides the extension $\widehat{X}$ on $\partial M \times [1, +\infty)$. The vector field $\widehat{X}$ is complete, i.e., its flow $\varphi_t$ is defined for every $t\in \R$; this justifies the name completion of a Liouville domain. We denote the completion of $(M, \lambda, X)$ by $(\widehat{M},\widehat{\lambda}, \widehat{X})$ and we call it \textit{Liouville manifold}. One can easily check that symplectizations from Examples \ref{ex:ball} and \ref{ex:codisk} are identified respectively with $(\R^{2n}, \sum 1/2 x_i dy_i - 1/2y_i dx_i)$ and $(T^*Q, \lambda_{can})$. Note that in the case of the ball $B^{2n}$, the radial coordinate $r$ is the norm squared of $(x_1,...,x_n,y_1,...,y_n)$ and not the norm. In the case of a codisk bundle $D^*_g Q$, the radial coordinate $r$ corresponds to $\|p \|^*_g$.

Given a smooth function $H: \widehat{M} \to \R$, one can associate to $H$ a vector field $X_H$ that is $\omega$-dual to $dH$, which means:
$$
i_{X_H} \omega = -dH.
$$
\begin{defn}
The function $H$ is a  \textit{Hamiltonian}, and the vector field $X_H$ is a \textit{Hamiltonian vector field}.
\end{defn}
Since we work with a class of symplectic manifolds $\widehat{M}$ which are not compact, the flow of $X_H$ might not be defined. However, the class of Hamiltonians that is of interest to us will have complete Hamiltonian vector fields $X_H$. We will also allow that our Hamiltonian functions are \textit{time dependent}: $H:\widehat{M} \times [0,1] \to \R$. It follows from Cartan's magic formula that the time-one map $\varphi^1_{H_t}$ of $X_{H_t}$ preserves the symplectic form. 

\begin{defn}\label{def:ham_linear_at_inf}
Hamiltonian $H: \widehat{M} \times [0,1] \to \R$ is called contact at infinity if there exists a smooth function $h_t :\partial M \to \R$ and $b \in \R$ such that
$$
H_t(x,r) = h_t(x) r +b, \text{ for } r \geq 1.
$$
\end{defn}
The function $h_t: \partial M \to \R$ is a \textit{contact Hamiltonian}, and the dynamics of $H_t$ is related to the contact isotopy of the contact vector field $X_{h_t}$. Recall  that for a contact Hamiltonian $h_t: \partial M \to \R$ on $(\partial M, \alpha)$, the contact vector field $X_{h_t}$ is uniquely determined by:
$$
\begin{aligned}
\alpha(X_{h_t}) &= h_t,\\
d h_t + i_{X_{h_t}} d \alpha &= dh_t(R_\alpha) \alpha,
\end{aligned}
$$
where $R_{\alpha}$ is the Reeb vector field uniquely determined by the equations from above for $h_t = 1$. From these equations, it easily follows that the Hamiltonian vector field $X_{H_t}$ of $H_t$ is given by:
\begin{equation}\label{eq:contact_at_inf}
	X_{H_t} = (X_{h_t}, -r dh_t(R_{\alpha})) \in T_x \partial M \times \R,
\end{equation}
for $r \geq 1$. In particular, if $h_t$ is constant, on each level $r = r_0 \geq 1$ the flow $X_{H_t}$ is a constant reparametrization of the Reeb flow. If a function $h: \partial M \to \R$ is independent of $t$, and strictly positive, then the flow of the Hamiltonian $H(x,r) = h(x) r + b$ can be seen as a Reeb flow on a hypersurface
$$
Y_h:= \left\{(x,r) \in \partial M \times (0, \infty) \mid r =\frac{1}{h(x)}\right\}.
$$  
The contact form on $Y_h$ is given by $\alpha_{Y_h} := \widehat{\lambda} \vert _{Y_h}$, and $(Y_h, \alpha_{Y_h})$ is strictly contactomorphic to $(\partial M,  \alpha / h)$. Every hypersurface $Y$ that is transverse to $\partial_r$ with the contact form given by the restriction of $\lambda$ is strictly contactomorphic to $(\partial M, \alpha/h)$ for some positive $h$. Such $Y$ gives rise to a different Liouville domain $(M_Y, \lambda, X) \subset \widehat{M}$ whose completion is again $\widehat{M}$. This perspective is useful since we want to allow ourselves to change a contact structure on $\partial M$. For more details about contact topology, we refer to \cite{Ge08}.

\begin{ex}
Let $H: \R^{2n} \to \R$, $H(x,y) =\frac{1}{2}\sum (x_i^2 + y_i^2)$. The Hamiltonian vector field is given by $X_H = (-y_1,...,-y_n, x_1,...,x_n)$ hence, the flow is 
$$
\varphi^t_H(z_1,...,z_n) = (e^{i t} z_1, ..., e^{it} z_n).
$$
\end{ex}

A smooth submanifold $L \subset \widehat{M}$ is \textit{Lagrangian} if $\omega\vert_{TL} = 0$ and $2\dim L = \dim \widehat{M}$.
\begin{defn}\label{def:lagr}
A Lagrangian submanifold $L \subset (\widehat{M}, \lambda)$ is \textit{exact} if there is a function $f_L : L \to \R$ such that $$\lambda\vert_{TL} = d f_L.$$ $L$ is \textit{cylindrical} if there exist a Legendrian\footnote{Submanifold $\Lambda$ of a contact manifold $Y, \alpha$ is Legendrian if $\alpha\vert_{T \Lambda} = 0$, and $2 \dim \Lambda + 1 = \dim Y$.} $\Lambda \subset \partial M$ such that
$$
L \cap (\partial M \times [1, \infty) ) = \Lambda \times [1, \infty) .
$$
\end{defn}

An equivalent way to define a cylindrical Lagrangian $L$ is to require that the Liouville vector field $X$ is tangent to $L$ for $r \geq 1$. 

If $L$ is exact and cylindrical, then $f_L$ is locally constant outside of a compact set. If in addition $\Lambda$ is connected, we can choose $f_L (x,r)= 0$ for $r \geq 1.$ One can deform the Liouville form $\lambda$  to $\lambda_{f_L} = \lambda - d f_L$, where we write $f_L$ for a smooth extension from $L$ to $\widehat{M}$. Hence $\lambda_{f_L} \vert_L = 0$. This changes the Liouville vector field, but it does not change the symplectic structure. From now on, we assume for simplicity that all Lagrangians $L$ satisfy $\lambda \vert L = 0$. 

\begin{ex}
Submanifold $L = \R^n \times \{ 0 \} \subset \R^{2n}$ is an exact cylindrical Lagrangian.
\end{ex}

\begin{ex}
Let $N \subset Q$ be a submanifold. The conormal bundle
$$
\nu^*N := \{ p \in T^*_N Q \mid p \vert_{TN} = 0 \},
$$
is an exact cylindrical Lagrangian. If $N := \{q\}$ we have that $\nu^*N = T^*_q Q$, and if $N:= Q$, $Q \cong \nu^*Q= \mathcal{O}_Q \subset T^*Q$ where $\mathcal{O}_Q$ is the zero section.  
\end{ex}
\begin{defn}
Map $J \in End(T \widehat{M})$ is an \textit{almost complex structure} if $$
J^2 = - \mathrm{id}.
$$
An almost complex structure $J$ is \textit{compatible} with $\omega$ if $\omega(\cdot, J \cdot)$ is a Riemannian metric. Since  $\omega(\cdot, J \cdot)$ is symmetric it is straightforward that
$$
\omega(J\cdot, J \cdot) = \omega(\cdot, \cdot).
$$
\end{defn}

It follows from \cite[Proposition 4.1.1]{MS17} that the space $\mathcal{J}_{comp}$ of compatible almost complex structure is non-empty and contractible. We will further assume that $J$ is of $SFT$-type, which means that
$$
\lambda \circ J = dr, \text{ for } r\geq 1.
$$

\section{Wrapped Floer homology}\label{sec:wrapped_floer}
The Floer homology was introduced in \cite{Fl88}. The main motivation was to solve the Arnold conjecture about the number of 1-periodic orbits of a Hamiltonian system on a closed symplectic manifold $(P, \omega)$. The original construction by Floer was later generalized and modified in many different setups. In the original setup, Floer considered a closed Lagrangian $L \subset P$ which satisfies $\pi_2(P,L) = 0$, and to a Hamiltonian $H: P \times [0,1] \to \R$ he associated a group $HF_*(L; H)$. He showed that $HF_*(L; H)$ does not depend on $H$, and in the case when $H$ is an extension of a $C^2$-small function $f: L \to \R$, $HF_*(L;H)$ is identified with the Morse homology $HM_*(L,f).$ This is related to the question of periodic orbits by setting Lagrangian $L$ to be the diagonal $\Delta \subset (P \times P, \omega \oplus (-\omega))$. 

The wrapped Floer homology is Floer homology for \textit{exact} Lagrangians, which are possibly non-compact inside of a non-compact symplectic manifold $(\widehat{M}, \widehat{\lambda})$. It was introduced in \cite{ASch06} for the case of a fiber $T_q^* N$ in the cotangent bundle $T^*N$, extended to the case of conormal bundle $\nu^*Q$ of a submanifold $Q \subset N$ in \cite{APS08} and generalized in \cite{AS10} to the class of exact cylindrical Lagrangians in Liouville manifolds. 

The main idea is to assign a group $HW_*(L_0, L_1)$ to a pair of exact cylindrical Lagrangians, where the generators of the chain complex are Hamiltonian $1$-chords with endpoints on $L_0$ and $L_1$ and the differential counts solutions of a PDE defined on $\R \times [0,1]$ and is assympotic to Hamiltonian chords at infinity. This PDE is obtained as a gradient of an action functional on the space of paths. The groups $HW_*(L_0,L_1)$ are graded with the Maslov index. The grading is explained in \S\ref{sec:maslov}. 

The condition $\pi_2(P,L) = 0$ in \cite{Fl88} was used to avoid bubbling of holomorphic disks and spheres so that one has the controlled behaviour of spaces of solutions of the equation used to define the differential. The differential is introduced in \S\ref{sec:diff}. In our situation, the assumption $\pi_2(P,L) = 0$ is replaced by the exactness of Lagrangians and of the symplectic manifold. The additional technical complication compared to \cite{Fl88} is the compactness of the space of solutions since both Lagrangians and symplectic manifold are non-compact. This is resolved by a certain type of maximum principle. These issues are addressed in \S\ref{sec:comp}. 

\subsection{The ungraded complex $CF(L_0, L_1 ; H)$}

One way to think about the Floer homology is as an infinite-dimensional analogue of Morse theory. Let $(\widehat{M}, \widehat{\lambda})$ be the Liouville manifold, and let  $L_0$ and $L_1$ be two proper Lagrangian submanifolds, with $\widehat{\lambda}\vert_{L_i} = 0$. The space on which we want to do the ``Morse theory" is the space of smooth paths with endpoints on $L_0$ and $L_1$:
$$
\mathcal{P}_{L_0, L_1} = \{x: [0,1] \to \widehat{M} \mid x(0) \in L_0 \text{ and } x(1) \in L_1\}.
$$
The role of a Morse function will have the action functional:
$$
\mathcal{A}_H (x) = \int x^* \widehat{\lambda} - \int_0^1 H_t(x(t)) dt,
$$
where $H_t$ is contact at infinity. 

\begin{lemma}\label{lemma:diff_action}
Critical points of $\mathcal{A}_H$ are Hamiltonian chords  $x:[0,1] \to \widehat{M}$ with endpoints on $L_i$.
\end{lemma}
\begin{proof}
Let $x_s$ be a path in $\mathcal{P}_{L_0, L_1}$ generated by $\frac{d}{ds} x_s \vert_{s=0}=\xi$. The path $x$ is a critical point of $\mathcal{A}_H$ if and only if for every deformation $\xi$:

$$
\begin{aligned}
0 &= d \mathcal{A}_H (\xi) = \frac{d}{ds} \mathcal{A}_H (x_s) \vert_{s=0}\\
&= \int x^*( d (i_{\xi}\widehat{\lambda}) + i_{\xi} d \widehat{\lambda}) - \int_0^1 dH_t(\xi (x(t)) dt = \\
&= \int_0^1\omega(\xi, x') dt +\int_0^1 \omega(X_{H_t}, \xi) dt + \widehat{\lambda} (\xi(x(1))) - \widehat{\lambda} (\xi(x(0)))\\
&= \int_0^1 \omega(\xi, x' - X_{H_t}(x)) dt.
\end{aligned}
$$

Here we have used generalization of Cartan's magic formula: $$\frac{d}{ds} x_s^* \widehat{\lambda}\vert_{s=0} = x_0^*( d (i_{\xi}\widehat{\lambda}) + i_{\xi} d \widehat{\lambda}),$$ and $\widehat{\lambda} (\xi(i)) = 0$ holds because $\xi(i) \in T_{x(i)} L_i$. Since $\omega$ is non-degenerate we have $x'(t) = X_{H_t} (x(t))$.
\end{proof}
We say that contact at infinity Hamiltonian $H_t$ is \textit{admissible} if the image $\varphi^1_{H_t} (L_0)$ of $L_0$ under the time one map $\varphi^1_{H_t}$ is transverse to $L_1$. Since Lagrangians $L_i$ and $X_H$ are invariant under the Liouville flow the assumption $\varphi^1_{H_t} (L_0) \pitchfork L_1$ implies that there are no Hamiltonian chords entirely contained in the region $r \geq 1$, in particular, the set of critical points $\crit (\mathcal{A}_H)$ is finite.

\begin{defn}
The ungraded Floer chain group is
$$
CF(L_0, L_1 ; H) := \bigoplus_{x \in \crit (\mathcal{A}_H)} \Z_2 \langle x \rangle.
$$
\end{defn}
\subsection{The Maslov index}\label{sec:maslov}
The main purpose of this section is to associate an integer $\mu(x)$ for each Hamiltonian chord $x\in \crit{\mathcal{A}_H}$. The assumptions on $c_1(\widehat{M})$ and the Maslov classes $\mu^{\Theta}_{L_i}$ from Theorem \ref{thm:main} are essential to have well defined $\Z$ grading. The approach to the grading is influenced by \cite{Se00, Au14}, we also refer to a foundational result \cite{RS93}. 

The Lagrangian Grassmannian in $\R^{2n}$ is
$$
\mathcal{L}(n):= \{L \leq \R^{2n} \mid L \text{ is a Lagrangian subspace} \}.
$$
One can show that $U(n)$ acts transitively on $\mathcal{L}(n)$, and the stabilizer is $O(n)$ hence
$$
\mathcal{L}(n) \cong U(n) / O(n).
$$
Furthermore, we know that $\det:U(n) \to S^1$ induces an isomorphism $\pi_1(U(n)) \cong \Z$. It follows that $\det^2: U(n) / O(n) \to S^1$ is well defined and that this map induces an isomorphism $\pi_1(\mathcal{L}(n)) \cong \Z$. For more details see \cite[\S2.3]{MS17} and \cite[\S10.1]{FVK18}. 

In order to assign an integer to a loop of Lagrangians planes in $T \widehat{M}$, we need to have $\det^2$ globally defined on each $T_p \widehat{M}$. This is equivalent to the triviality of the square of the complex determinant line bundle $(\bigwedge^n_{\C} T\widehat{M})^{\otimes 2}$. Since for a complex vector bundle $E$ we have $c_1(E) = c_1 (\det (E))$, the triviality of $(\bigwedge^n_{\C} T\widehat{M})^{\otimes 2}$ is equivalent to $2c_1(\widehat{M}) = 0$. 

\begin{ex}
If the tangent bundle $TM$ of a symplectic manifold has a Lagrangian subbundle $F \subset TM$, then $2c_1(M) = 0$, if we assume that $F$ is orientable we get that $c_1(M) = 0.$ In particular, for $M = T^*N$ we have $2c_1(T^*N)=0$, and if $N$ is orientable $c_1(T^*N) = 0.$ More generally, if $M$ is foliated by Lagrangian submanifolds, then the foliation induces a Lagrangian subbundle.
The triviality is obtained by using the section of $(\bigwedge_{\R}^n F)^{\otimes 2}$, or respectively of $\bigwedge_{\R}^n F$.
\end{ex}

Under the assumption $2c_1(M)=0$, one can assign a \textit{Maslov class} $\mu^{\Theta}_L \in H^1(L)$ to every Lagrangian submanifold $L \subset M$ and a non-vanishing section $\Theta$ of $(\bigwedge^n_{\C} T\widehat{M})^{\otimes 2}$.  Define $\mu_L^{\Theta} \in H^1(L) \cong [L; S^1],$
by $$\mu_L^{\Theta} (p) = Arg (\Theta(p), \Omega_L \otimes \Omega_L),$$
where for $u, v \in \C^*$, $Arg(u,v)$ is the unique element $w \in S^1$ such that $u$ and $w v$ are positively proportional, and $\Omega_L \in \bigwedge_{\R}^n T_p L$ is any non-zero element. 

The relative Chern class $2c_1(M, L) \in H^2(M,L)$ is defined as the Poincar\'e dual of a zero set of a certain section. Let $s:M \to (\bigwedge^n_{\C} T M)^{\otimes 2}$ be a section, such that $s\vert_L \in (\bigwedge^n_{\R} T L)^{\otimes 2}$. If $s$ is transverse to the zero section, then 
$$
2c_1(M,L) := \mathrm{PD}(s^{-1}(0)) \in H^2(M,L).$$ 
From the definition, it follows that $2 c_1(M,L)$ is mapped to $2c_1(M)$. Hence, the assumption $2c_1(M)=0$ in Theorem \ref{thm:main} is redundant in the case $L=L_0 = L_1$. 

From now on, we fix a reference path $\gamma_{[x]}$ for each connected component $[x]$ of $\mathcal{P}_{L_0, L_1}$. We also fix a section $\Gamma_{[x]}$ of $\gamma_{[x]}^* \mathcal{L}(TM)$, where $\mathcal{L}(TM)$ is a fiber bundle over $M$ with fibers which are all Lagrangian subspaces of $T_p M$, hence it is identified with $\mathcal{L}(n)$. We also require that $\Gamma_{[x]}(i) = T_{\gamma_{[x]} (i)} L_i$. 

Given an admissible Hamiltonian $H_t$, we want to assign an integer $\mu(x)$ to every Hamiltonian chord $x \in \crit \mathcal{A}_{H}$. Let $\gamma_{[x]}$ be a reference path that is in the same connected component of $\mathcal{P}_{L_0, L_1}$ as $x$, and let $u:[0,1] \times [0,1] \to M$ be a path in $\mathcal{P}_{L_0, L_1}$ such that $u(0,t)= \gamma_{[x]}(t)$, and $u(1,t) = x(t)$. There is a symplectic trivialization of $\Phi:u^* TM \cong [0,1] \times [0,1] \times \R^{2n}$, we choose the trivialization so that on the boundary of the square $[0,1]\times[0,1]$ it induces the section of $(\bigwedge^n_{\C} TM)^{\otimes 2}$ which is equal to $\Theta$. 

In order to define a loop of Lagrangians, we need one more auxiliary choice. For two transverse Lagrangian spaces $L_i \leq \R^{2n}$ there is a symplectic matrix $A\in Sp(n)$ such that $A L_0 = \R^n \times \{ 0\}$, and $A L_1 = \{0\} \times \R^n$. A \textit{ canonical short path } is the path of Lagrangians $L_t := A^{-1}( (e^{-i \pi t /2} \R)^n)$ for $t \in [0,1]$. Now, define a loop of Lagrangians

$$
\alpha(s)= \begin{cases}\Phi(T_{u(5s,0)} L_0), &s \in[0,1/5]\\
					\Phi( d \varphi^{5s-1}_H T_{x(0)} L_0), &s \in [1/5, 2/5]\\
					A_{x(1)}(5s-2), &s \in [2/5, 3/5]\\
					\Phi(T_{u(4-5s,1)} L_1), &s\in [3/5, 4/5]\\
					\Phi(\Gamma_{[x]}(5-5s)), &s\in [4/5, 1].
					 \end{cases}
$$
Here $A_{x(1)}$ is the canonical short path from $\Phi( d \varphi^{1}_H T_{x(0)} L_0)$ to $\Phi(T_{x(1)} L_1)$.

\begin{defn}
The Maslov index of $x \in \crit \mathcal{A}_{H}$ is given by
$$
\mu(x):= \mathrm{det}^2(\alpha).
$$
\end{defn}
Since we have forced a relative homotopy class of trivializations, we have that this does not depend on the choice of trivialization. It is left to show that the index $\mu(x)$ is independent of the choice of $u:[0,1] \times [0,1] \to M$ which joins $\gamma_{[x]}$ and $x$. This follows from the assumption that the Maslov classes $\mu^\Theta_{L_i}$ are exact. Indeed, two different choices $u_1$ and $u_2$ of relative homotopies are creating loops $\beta_i(t)=u_1(i,2t) \# u_2(i,2-2t) \in L_i$. Since trivializations over $u_1$ and $u_2$ were chosen to be compatible with $\Theta$, we have that 
$$\mathrm{det}^2(\Phi(\beta_i)) = \mu^{\Theta}_{L_i} ([\beta_i] ) = 0,$$ 
where we have used the notation $\beta_i$ both as a loop on $L_i$, and loops of Lagrangians in $\mathcal{L}(TM)$.

The alternative way would be to define $$\mu(x):= \mu_{RS}(\Lambda_0, \Lambda_1)-\frac{n}{2},$$ where  $\mu_{RS}(\Lambda_0, \Lambda_1)$ is the Robbin-Salamon index (\cite{RS93}) of two paths $\Lambda_0(t),\Lambda_1(t)$ of Lagrangians given by
$$
\begin{aligned}
\Lambda_0(t) &= \begin{cases} \Phi(\Gamma_{[x]}(1-3t)), &t\in [0, 1/3]\\
					\Phi(T_{u(3t-1,0)} L_0), &t \in[1/3,2/3]\\
					\Phi( d \varphi^{3t-2}_H T_{x(0)} L_0), &t \in [2/3, 1]
					 \end{cases}\\
\Lambda_1(t) &= \Phi(T_{u(t,1)} L_1)	.	
\end{aligned}			 
$$
The advantage of this approach is that it is well-defined for degenerate Hamiltonians. Below, we state some properties of the Robin-Salamon index $\mu_{RS}$ without the proof. 

\begin{thm}
\begin{itemize}
\item \textbf{(Naturality)} For a path $\Psi: [0,1] \to Sp(n)$, $$\mu_{RS}(\Psi \Lambda_0, \Psi \Lambda_1) = \mu_{RS}(\Lambda_0, \Lambda_1).$$
\item \textbf{(Concatenation)} For $c \in [0,1]$, $$\mu_{RS}(\Lambda_0, \Lambda_1) = \mu_{RS}(\Lambda_0 \vert_{[0,c]}, \Lambda_1\vert_{[0,c]}) + \mu_{RS}(\Lambda_0 \vert_{[c,1]}, \Lambda_1\vert_{[c,1]}).$$
\item \textbf{(Product)} $\mu_{RS}(\Lambda_0' \oplus \Lambda_0 '', \Lambda_1' \oplus \Lambda_1'') = \mu_{RS}(\Lambda_0', \Lambda_1')+\mu_{RS}( \Lambda_1'',  \Lambda_1'').$
\item \textbf{(Localization)} If $\Lambda_1(t) = V:=\R^n \times \{0\}$, and $$\Lambda_0(t) = \{(x, A(t) x) \mid x\in \R^n\},$$ where $A(t)$ is a path of symmetric matrices then

$$
\mu_{RS}(\Lambda_0, \Lambda_1) = \frac{1}{2} \mathrm{sign} A(1) - \frac{1}{2} \mathrm{sign} A(0).
$$
\item \textbf{(Homotopy)} Paths $\Lambda_0, \Lambda_1: [0,1] \to \mathcal{L}(n)$ with $\Lambda_0(0) = \Lambda_1(0)$, and $\Lambda_0(1) = \Lambda_1(1)$ are homotopic if and only if $\mu_{RS}(\Lambda_0, V) = \mu_{RS}(\Lambda_1, V)$.
\end{itemize}
\end{thm}

\begin{ex}\label{ex:index}
Let $L_i= \R \times \{0\}$, and $H_k(z) = \frac{2k+1 }{4} \pi|z|^2$,
$$
X_{H_k}(x,y) = (-\partial_y {H_k}, \partial_x {H_k}) = \left(- \left(k+\frac{1}{2}\right)  \pi y, \left(k+\frac{1}{2}\right) \pi x\right)  = i \left(k + \frac{1}{2}\right)\pi z .
$$
The flow is given by
$$
\varphi^t_{H_k}(z) = e^{i (k+1/2) \pi t} z.
$$
Set $\gamma_{[x]} (t) = 0$, and $\Gamma_{[x]}(t) = \R \times \{0\}$. For every $k\in \N$ the unique chord is given by $x_k(t) = 0$, it is non-degenerate since $\varphi^1_{H_k}(\R \times \{0\}) = \{0\} \times \R$, which intersects $\R \times \{0\}$ transversally. One easily verifies that
$$
\mu(x) = k,
$$
since $\mu$ measures the number of completed half-rotations.
\end{ex}
A similar calculation will appear later in Section \ref{sec:inv_handle}.

\subsection{Differential}\label{sec:diff}
In this section we define the differential $$d:CF_k(L_0, L_1;H_t) \to CF_{k-1}(L_0, L_1;H_t).$$ The idea is to study the space of solutions of a PDE that replaces the gradient flow in the Morse case. Consider an admissible Hamiltonian $H_t$, an almost complex structure $J_t$.  On the tangent space of $\mathcal{P}_{L_0, L_1}$, there is an inner product induced by $\omega$ and $J_t$. For $x \in \mathcal{P}_{L_0, L_1}$, and two  sections $\xi, \eta$ of $x^* T \widehat{M}$ set
$$
\langle \xi, \eta \rangle_{L^2} = \int_0^1 \omega(\xi(x(t)), J_t \eta(x(t)) dt.
$$

Recall from Lemma \ref{lemma:diff_action} that:
$$
d \mathcal{A}_{H_t}(\xi) = \int_0^1 \omega(\xi, x' - X_{H_t}(x)) dt,
$$
 hence we get  $\nabla_{J_t} \mathcal{A}_{H_t} = -J_t (x' - X_{H_t}) $. The gradient equation of $\mathcal{A}_{H_t}$ is  
$$
\frac{d}{ds} u(s) = -J_t \left( u(s)' - X_{H_t}(u(s)) \right),
$$
for $u:\R \to \mathcal{P}_{L_0, L_1}$. This ODE is poorly behaved on the infinite-dimensional space $\mathcal{P}_{L_0, L_1}$. For instance, it does not induce a flow on $\mathcal{P}_{L_0, L_1}$, see \cite[Remark 11.3.1]{PR14}. However, one can rewrite it as a PDE on the finite-dimensional manifold $\widehat{M}$:
\begin{equation}\label{eq:floer_eq}
	\partial_s u + J_t (\partial_t u - X_{H_t}) = 0.
\end{equation}
 This equation is called \textit{Floer's equation}. 
 \begin{defn} For $x_-, x_+ \in \crit(\mathcal{A}_{H_t})$ the moduli space $\mathcal{M}(x_-,x_+,H,J)$ is the set of $u:\R \times [0,1] \to \widehat{M}$ which satisfy Floer's equation and have the asymptotic conditions $\lim\limits_{s \to \pm \infty} u(s,t) = x_{\pm}(t)$. 
 \end{defn}
 Note that before we wrote Floer's equation, $u$ was a map to $\mathcal{P}_{L_0, L_1}$, so $u(s,t)$ satisfies the boundary conditions $u(s,i) \in L_i$. Later on we write $\mathcal{M}(x_-,x_+):=\mathcal{M}(x_-,x_+,H,J)$ when $H$ and $J$ are clear from the context.
 
 The differential will be the count of ``unparametrized" elements in $\mathcal{M}(x_-,x_+)$ when $\mu(x_+) - \mu(x_-) = 1$.  ``Unparametrized" means that we divide $\mathcal{M}(x_-,x_+)$ by the action of $\R$, which acts by translations. This is possible since Floer's equation is invariant under translation in the $s$ direction. Hence, the goal is to show that $\mathcal{M}(x_-,x_+)$ is a manifold that can be compactified by broken trajectories. Furthermore, all possible breakings appear in this compactification. These statements are respectively explained in sections about transversality, compactness, and gluing. All these statements rely on ground-breaking ideas from \cite{Gr85} and \cite{Fl88}. To apply the ideas from the case of closed manifolds, one needs to show that there exists a compact set $K \subset \widehat{M}$ such that all the elements from $\mathcal{M}(x_-,x_+)$ have images inside $K$. This is known as the \textit{maximum principle}. The original proof appears in \cite[Lemma 1.8]{Vi99} for the case of closed Hamiltonian loops, and here we present it. The same proof in the Lagrangian setting appears, e.g., in \cite[Proposition 2.8]{BCS24} or \cite[\S D.4]{Ri13}.  It covers the Hamiltonians, which are of the form $H_s(x,r) = f_s(r)$ outside of the compact set $\{ r\geq 1\}$. In Lemma \ref{lemma:no_escape_2}, we will show it for more general Hamiltonians following \cite{Fa20}. For further generalizations of the maximum principle that covers Hamiltonians which are contact at infinity see \cite[Theorem 1.1]{MU19} and \cite[\S2.2.5]{BC24}. 
 
 \begin{prop}\label{prop:maximum}
 Let $H_s$ be a Hamiltonian of the form $H_s(x,r) = f_s(r)$ on $\{r\geq 1\}$. If $\partial_s \partial_r f \leq 0$ then there exist a compact set $K$ such that all elements $u \in \mathcal{M}(x_-,x_+)$ have images $\mathrm{Im}(u)$ contained in  $K$.
 \end{prop}
 \begin{proof}
 We will show that $\rho(s,t) := r \circ u (s,t)$ satisfies the strong maximum principle, and hence, it must obtain maximum at the asymptotics $x_{\pm}$. 
 $$
 \begin{aligned}
 \partial_t \rho &= d r ( \partial_t u) = \lambda (J_t \partial_t u) = \lambda(-\partial_s u + J_t X_{H_t} ) = -\lambda(\partial_s u) \\
 \partial_s \rho &= d r (\partial_s u ) = \lambda (J_t \partial_s u) = \lambda( \partial_t u - X_{H_t}) =   \lambda( \partial_t u) - \rho \partial_r f_s(\rho).
 \end{aligned}
 $$
 We have used that $u$ satisfies Floer's equation $\partial_s u + J_t (\partial_t u - X_{H_t})=0$, that $\lambda \circ J_t = dr$ and that $X_H = \partial_r f_s(r) R_{\alpha}$. Further, we have
 
 $$
 \begin{aligned}
 |\partial_s u|^2_J:=\omega (\partial_s u, J_t \partial_s u) &= \omega(\partial_s u, \partial_t u - X_{H_t})\\
 &= \omega(\partial_s u, \partial_t u) - dH_t(\partial_s u)\\
 &= \omega(\partial_s u, \partial_t u) - \partial_r f_s(r) \partial_s \rho\\
 &= \partial_s \lambda(\partial_t u) - \partial_t \lambda(\partial_s u) - \partial_r f_s(r) \partial_s \rho,
 \end{aligned}
 $$
 and
 $$
 \begin{aligned}
 \Delta \rho &= \partial_s^2 \rho + \partial_t^2 \rho\\
 &= \partial_s \lambda(\partial_t u) - \partial_s \rho\cdot \partial_r f_s(\rho)  - \rho \cdot \partial_s \partial_r f_s(\rho) - \rho \cdot \partial_r^2 f_s(\rho) \cdot \partial_s \rho - \partial_t \lambda( \partial_s u)\\
&=\partial_s \lambda(\partial_t u) - \partial_t \lambda( \partial_s u) - \partial_s \rho \cdot \partial_r f_s(\rho)  - \rho \cdot \partial_s \partial_r f_s(\rho) - \rho \cdot \partial_r^2 f_s(\rho) \cdot \partial_s \rho\\
&= |\partial_s u|^2_J  -\rho \cdot \partial_s \partial_r f_s(\rho) - \rho \cdot \partial_r^2 f_s(\rho)  \cdot \partial_s \rho.
 \end{aligned}
 $$
 Set an elliptic operator:
 $$
 L v :=  \Delta v +\rho \cdot \partial_r^2 f_s(v)  \cdot \partial_s v,
 $$
 it follows from the estimates above that $$L\rho \geq |\partial_s u|^2_J  -\rho \cdot \partial_s \partial_r f_s(\rho),$$
 and from the assumption $\partial_s \partial_r f_s \leq 0$ we get $L \rho \geq 0$. Hence, the solutions satisfy the strong maximum principle and can not attain the maximum in the interior. There is a possibility that $\rho$ can attain maximum at the boundary, but then it follows from Hopf's Lemma  (see \cite[\S 6.4.2]{Ev10}) that $\partial_t \rho (s_0, 1) >0$ if the maximum is at $(s_0, 1)$, or $\partial_t \rho (s_0, 0) <0 $ if the maximum is at $(s_0, 0)$. This is impossible because of Lagrangian boundary conditions on $u$, and since $\lambda\vert_{L_i} = 0$. Indeed, one has $\partial_t \rho =-\lambda(\partial_s u)$ and $u(s,i) \in L_i$. 
 \end{proof}
 
 Now we list the properties of the moduli space $\mathcal{M}(x_-, x_+)$. Denote by $\Bar{\mathcal{M}}(x_-, x_+)$ the quotient of  $\mathcal{M}(x_-, x_+)$ by $\R$ action.
 
 \begin{thm}\label{thm:moduli_space}
 There exist a set of almost complex structures $\mathcal{J}_{reg} \subset \mathcal{J}$, of the second Baire category such that for $J_t \in \mathcal{J}_{reg}$, $\mathcal{M}(x_-, x_+, H_t, J_t)$ is a smooth manifold of dimension $\mu(x_+) - \mu(x_-)$. The smooth manifold $\Bar{\mathcal{M}}(x_-, x_+)$ is pre-compact. If $\mu(x_+) - \mu(x_-)=1$ it is a finite set of points, and if $\mu(x_+) - \mu(x_-)=2$:
 $$
 \partial \Bar{\mathcal{M}}(x_-, x_+) \cong \bigcup_{\substack{y, \\ \mu(y) - \mu(x_-) = 1}} \Bar{\mathcal{M}}(x_-, y) \times \Bar{\mathcal{M}}(y, x_+).
 $$
 \end{thm}
 
We will elaborate on this theorem in Sections \ref{sec:trans}, \ref{sec:comp}, and \ref{sec:gluing}. For more details see \cite{MS12, AD14}.
 
 Now, define $d:CF_k(L_0, L_1;H_t) \to CF_{k-1}(L_0, L_1;H_t)$ on generators $x_+$ of index $\mu(x_+)=k$ by:
 \begin{equation}\label{eq:diff}
 d x_+ = \sum_{\mu(x_-) = k-1} \#_2 \Bar{\mathcal{M}}(x_-, x_+) x_-,
 \end{equation}
 and extend to $CF_k(L_0, L_1;H_t)$ by linearity. Since $\mu(x_+) - \mu(x_-) = 1$, it follows from Theorem \ref{thm:moduli_space} that $\Bar{\mathcal{M}}(x_-, x_+)$ is finite set of points, hence $d$ is well defined. 
 
 \begin{lemma}\label{lemma:differential}
 The map $d$ is a differential.
 \end{lemma}
 \begin{proof}
 Take $z\in CF_k(L_0, L_1;H_t)$:
 $$
 \begin{aligned}
 d^2 z &= d\left( \sum_{y } \#_2 \Bar{\mathcal{M}}(y, z) y \right)\\
  &= \sum_{y}\Bar{\mathcal{M}}(y, z) \left( \sum_{x} \#_2 \Bar{\mathcal{M}}(x, y) \#_2  x\right)\\
 &=\sum_x \sum_y  \#_2\Bar{\mathcal{M}}(x, y) \cdot \#_2\Bar{\mathcal{M}}(y, z) x \\
 &= \sum_x  \sum_y \#_2 (\Bar{\mathcal{M}}(x, y)  \times  \Bar{\mathcal{M}}(y, z)) x\\
 &= \sum_x \#_2 \partial \Bar{\mathcal{M}}(x,z) x = 0.
 \end{aligned}
 $$
 Here we have used the structure of the boundary of $\Bar{\mathcal{M}}(x,z)$ from  Theorem \ref{thm:moduli_space} in the case $\mu(z) - \mu(x) = 2$. We have also used that the cardinality of the boundary of a compact one-dimensional manifold is even.
 \end{proof}
 
\subsubsection{Transversality and the Index theorem}\label{sec:trans}
In this section, we briefly cover the setup for showing that the moduli space $\mathcal{M}(x_-,x_+)$ has the structure of a smooth manifold. This is a consequence of the infinite-dimensional implicit function theorem for Fredholm maps, together with the Sard-Smale theorem. However, this does not hold unconditionally; $\mathcal{M}(x_-,x_+)$ is a smooth manifold for the almost complex structures $J_t$ that are \textit{regular}. We denote the set of regular almost complex structures $\mathcal{J}_{reg}$. The goal is to explain that this set is ``large", i.e., it is a countable intersection of open and dense sets\footnote{This is the definition of a set of the second Baire category.}.

The strategy is as follows: we want to consider the solutions of Floer's equation as a zero set of a section $\mathcal{F}_{H,J}$ of an appropriate Banach bundle, then we want to show that the vertical derivative $D_u$ of such a map is Fredholm. An operator between Banach spaces is Fredholm if it has a closed image and finite-dimensional kernel and cokernel. A priori, $\mathcal{F}_{H,J}$ does not need to be transverse to the zero section (or equivalently $D_u$ does not need to be surjective), so we can not appeal to the infinite-dimensional implicit function theorem that the zero set is a smooth manifold. In the case $D_u$ is surjective, the tangent space of the moduli space at $u$ is given by $\ker D_u$. Put differently, the expected dimension of the connected component of $\mathcal{M}(x_-,x_+)$ that contains $u$ is given by the \textit{Fredholm index} $\mathrm{ind} D_u := \dim \ker D_u - \dim \mathrm{coker} D_u$. Each $u \in \mathcal{M}(x_-,x_+)$ canonically assigns the \textit{relative} Maslov index $\mu_u(x_-, x_+)$, that in the presence of globally defined grading satisfies
$$
\mu_u(x_-, x_+) = \mu(x_+) - \mu(x_-).
$$
To give some context, the domain of $\mathcal{F}_{H,J}$ is the $W^{1,p}$-completion of smooth maps $u:\R \times [0,1] \to W$, which satisfy Lagrangian boundary condition, and that converge sub-exponentially, uniformly in $t$, to $x_{\pm}$ when $s\to \pm \infty$. Denote the space of smooth maps with $C^{\infty}(x_-, x_+)$, and its $W^{1,p}$-completion with $\mathcal{W}^{1,p}(x_-, x_+)$. It is important to choose $p>2$ so that each $u \in \mathcal{W}^{1,p}(x_-, x_+)$ is continuous. Now consider the tangent space of the Fr\'echet manifold $C^{\infty}(x_-, x_+)$, whose fiber at $u \in C^{\infty}(x_-, x_+)$ is identified with the sections $\xi$ of $u^* TW$ that are tangent to the Lagrangians $L_i$ on the ends, and exponentially converge to $0$ in the ends. Denote by $\mathcal{L}^p(x_-, x_+)$ the fiber-wise $L^p$-completion of the tangent space to $C^{\infty}(x_-, x_+)$. \textit{Floer's map} is given by

\begin{equation}\label{eq:Floer_section}
	\begin{aligned}
		\mathcal{F}_{H,J}: \mathcal{W}^{1,p}(x_-, x_+) &\to \mathcal{L}^p(x_-, x_+) \\
		u &\mapsto \frac{\partial u}{\partial s} + J_t \left(\frac{\partial u}{\partial t} - X_{H_t}(u)\right).
	\end{aligned}
\end{equation}

After using the symplectic trivialisation of $u^*TW \cong \R \times [0,1] \times \R^{2n}$ (such that $T_{u(s,i)}L_{i} \cong \R^n \times \{0\}$), the vertical derivative of $\mathcal{F}_{H,J}$ at $u \in \mathcal{F}_{H,J}^{-1}(0)$ is identified with
$$
D_u \xi = \partial_s \xi + J_0\partial_t \xi + S(s,t),
$$
where $\xi \in W^{1,p}(\R \times [0,1], \R\times \{0,1\}; \R^{2n}, \R^n \times \{0\})$, and $S(s,t)$ converges uniformly to a path of symmetric matrices $S_{\pm}(t)$ when $s \to \pm \infty$.  The map $S_{\pm}(t)$ is determined by $\frac{d}{dt} A(t)=J_0 S_{\pm}(t)\circ A(t)$, where $A(t)$ is identified with $D \varphi_{H_t}^t(x_{\pm}(t))$ using the trivialization. For details on the form of $D_u$ after trivialization see \cite[\S2.2]{Sa97}, or for a different approach using the Whitney embedding see \cite[\S8.4]{AD14}.

The following lemma from functional analysis will be helpful in proving the Fredholm property of $D_u$.

\begin{lemma}\label{lemma:fredholm operator}
Let $D:X \to Y$ be a bounded operator and $K: X \to Z$ a compact operator where $X, Y$ and $Z$ are Banach spaces, if there is $c>0$ such that for every $x \in X$ we have
$$
\|x\|_X \leq c (\|Dx\|_Y + \|K x\|_Z),
$$
then $D$ has a finite-dimensional kernel and closed image.
\end{lemma}
\begin{proof}
See \cite[Lemma A.1.1.]{MS12}.
\end{proof}
The Fredholm property for $D_u$ will follow after showing that both $D_u$ and its formal adjoint $D^*_u$ satisfy: 
\begin{equation} \label{eq:l_p-estimate}
\| \xi\|_{W^{1,p}} \leq c (\|D \xi\|_{L^p} + \|\xi\|_{L^p(-T,T)}),
\end{equation}
since the inclusion $L^p ([-T,T] \times [0,1]) \to L^p(\R \times [0,1])$ is compact. We will not provide all the details, however we will give proof of some intermediate steps.

\begin{lemma}\label{lemma:l_p-estimate}
There is a constant $c>0$ such that
$$
\| \xi\|_{W^{1,p}} \leq c (\|D \xi\|_{L^p} + \|\xi\|_{L^p}).
$$
\end{lemma}
\begin{proof}
Recall the Calderon-Zygmund inequality (see \cite[Corollary B.2.8.]{MS12}) for maps $u:\Omega \to \R$, where $\Omega \subset \R^m$:
$$
\sum \|\partial_i \partial_j u\|_{L^p} \leq c_1 \| \Delta u \|_{L^p}.
$$
Here $\Delta:= \frac{\partial^2}{\partial x_1^2}+ \cdots+\frac{\partial^2}{\partial x_m^2}$ is the standard Laplacian. 

Now, let's start with $\xi:= \partial_s u - J_0 \partial_t u$ (here, we are implicitly using the fact that the operator $u \mapsto \partial_s u - J_0 \partial_t u$ is bijective). Note that one can decompose $\Delta$ as $(\partial_s + J_0 \partial_t) (\partial_s - J_0 \partial_t)$.
\begin{equation*}
\begin{aligned}
\| \xi\|_{W^{1,p}} &= \| \partial_s \xi\|_{L^p} + \| \partial_t \xi\|_{L^p} + \|\xi\|_{L^p}\\ 
			&\leq 2 ( \| \partial_s \partial_s u \|_{L^p} + \| \partial_t \partial_s u \|_{L^p} + \| \partial_t \partial_t u \|_{L^p} + \|\xi\|_{L^p})\\
			&\leq 2c_1( \|\Delta u\|_{L^p} + \|\xi\|_{L^p}).
\end{aligned}
\end{equation*}
On the other hand, note that $D_u \xi - S\xi = \Delta u$, hence we have
$$
\| \xi\|_{W^{1,p}} \leq 2c_1 ( \|D_u \xi\|_{L^p} + \|S \xi\|_{L^p} + \|\xi\|_{L^p}).
$$
Since the $C^0$ norm of $S$ is bounded we get
$$
\| \xi\|_{W^{1,p}} \leq c ( \|D_u \xi\|_{L^p} + \|\xi\|_{L^p}),
$$
where $c = 2c_1 (1+ \|S\|_{C^0})$.
\end{proof}
The estimate (\ref{eq:l_p-estimate}) will hold if the symplectic matrices $\psi_{\pm}(1)$ which are determined by $d/dt \psi_{\pm} (t) = J_0 S_{\pm}(t) \psi_{\pm}(t)$ satisfy $\psi_{\pm}(1) (\R^n \times \{0\}) \pitchfork \R^n \times \{0\}$. This will follow from Lemma \ref{lemma:l_p-estimate} and
\begin{lemma}
If $S = S(t)$ does not depend on $s$, and if $\psi(1) (\R^n \times \{0\}) \pitchfork \R^n \times \{0\}$ then the operator
$$
D \xi = \partial_s \xi + J_0 \partial_t \xi +S(t) 
$$
is bijective for $1<p<\infty$.
\end{lemma}
\begin{proof}
We will show that the operator $A: W^{1,p}([0,1]) \to L^p([0,1])$ given by $A= J_0 \partial_t + S$ is bijective. This follows from the transversality condition and the theory of ODE. 

\textit{$A$ is injective.} Let $u:[0,1] \to \R^{2n}$ be a map such that $u(i) \in \R^n \times \{0\}$, and
$$
\partial_t u = J_0 S(t) u(t).
$$
Then $u(t) = \exp( \int_0^t J_0 S(\tau) \mathrm{d}\tau) u_0$ for some $u_0 \in \R^n \times \{0\}$. Now since $u(1) = \psi u_0$, and $\psi (\R^n \times \{0\}) \pitchfork \R^n \times \{0\}$, we get $u_0 = 0$.

\textit{$A$ is surjective.} Instead of doing a variation of the constant for the homogeneous equation we will give a slightly different argument in the spirit of the Fredholm theory. The operator $$J_0\frac{d}{dt}: W^{1,p} \to L^p$$ has $n$-dimensional kernel given by the constant maps, and its formal adjoint is equal to $J_0\frac{d}{dt}$, hence its Fredholm index is $0$. The operator given by the multiplication with $S$ is a compact operator $S\cdot : W^{1,p} \to L^p$, since the inclusion $W^{1,p} \to L^p$ is compact. Now, if $F$ is a Fredholm operator, and $K$ compact, then $F + K$ is Fredholm and $\mathrm{ind} F = \mathrm{ind} (F+K)$ (see \cite[\S A.1]{MS12}. Since $A$ is injective, and has index $0$, it must be surjective. For the rest of the proof see \cite[Lemma 2.4]{Sa97}.
\end{proof}
   
   Now we will sketch the proof of the index theorem:
   
  \begin{thm}
  The Fredholm index of $D_u$ is equal to $\mu(x_+) - \mu(x_-)$.
  \end{thm}
  \begin{proof}
  Here, we will focus just on the case when $n=1$ and $t$-independent family $S(s)$ which satisfies
  $$S(s)=\begin{bmatrix} a_{\pm} & 0 \\ 0 & a_{\pm}\end{bmatrix},$$
 for $\pm s \geq s_0$. Using the homotopy property of the Fredholm index, and the homotopy property of the Maslov index this is enough since every integer can obtained as a Maslov index of paths determined by such $S$. Set $w(s,t) = \exp(\int_0^s S(\sigma) \mathrm{d} \sigma )v$. One easily verifies that if $v$ is in the kernel of $D_u$, then $w$ is holomorphic with the same boundary conditions as $v$, i.e., $w(s,i) \in \R \times \{0\}$. Hence, we have 
 $$
 w(s,t) = \sum_{\ell \in \Z} c_{\ell} e^{\ell \pi (s + it)}.
 $$ 
 From the boundary conditions we get that the imaginary part $b_{\ell}$ of $c_{\ell} = a_{\ell} + i b_{\ell}$ vanishes. Hence, for $s\geq s_0$ our solution $v$ is of the form:
 $$
 v(s,t)= \sum_{\ell} a'_{\ell} e^{s(\pi \ell - a_+)} e^{i\ell \pi t}.
 $$
 Similarly, for $s \leq - s_0$ we have:
  $$
 v(s,t)= \sum_{\ell} a''_{\ell} e^{s(\pi \ell - a_-)} e^{i\ell \pi t}.
 $$
 So, in order that $v$ belongs to $W^{1,p}$ we must have $a_- < \ell \pi < a_+$. Consequently, the dimension of the kernel is
 $$
 \dim (\ker D_u)= \# \{\ell \mid a_- < \ell \pi < a_+\}.
 $$
 On the other hand, by analogous argument the kernel of $D_u^* = -\partial_s + J_0 \partial_t + S(s)$ satisfies
 $$
\dim (\ker D^*_u)= \# \{\ell \mid a_+ < \ell \pi < a_-\}.
 $$
 Since $\mathrm{coker} D_u \cong \ker D_u^*$ we get that the index of $D_u$ is given by
 $$
 \mathrm{ind} D_u = \left\lfloor \frac{a_+}{\pi} \right\rfloor - \left\lfloor \frac{a_-}{\pi} \right\rfloor.
 $$
 From a similar calculation as in Example \ref{ex:index}, we get that
 $$
 \mathrm{ind} D_u = \mu(S_+) - \mu(S_-).
 $$
  \end{proof}
  
  In order to apply the Implicit function theorem (\cite[Theorem A.3.3.]{MS12}), we need that $D_u$ is surjective. In order to achieve this, we will consider the universal Floer map
  \begin{equation}\label{eq:Floer_section_universal}
	\begin{aligned}
		\mathcal{F}_{H}: \mathcal{W}^{1,p}(x_-, x_+) \times \mathcal{J}_{\epsilon}(J')&\to \mathcal{L}^p(x_-, x_+) \\
		(u,J) &\mapsto \frac{\partial u}{\partial s} + J_t \left(\frac{\partial u}{\partial t} - X_{H_t}(u)\right).
	\end{aligned}
\end{equation}
The space $\mathcal{J}_{\epsilon}(J')$ is a Banach manifold of smooth almost complex structures $J$ that satisfy $\|J-J'\| = \sum \epsilon_k \|J-J'\|_{C^k} < \infty$, for appropriately chosen $\epsilon_k$. Since the tangent space of $\mathcal{J}_{\epsilon}(J')$ is very large, one can achieve that this map is transverse to the zero section. This essentially follows from the fact that the tangent space to the space of linear complex structures on $\R^{2n}$ that are compatible with $\omega_{st}$ is large enough, together with the fact that elements of the moduli space generically have $s$-injective points (see \cite[\S 8.6]{AD14}). To be more precise, the tangent space at $J_0$ is given by $T_{J_0} \mathcal{J}(\R^{2n}) = \{Y \in M_{2n}(\R) \mid Y=Y^T = J_0 Y J_0\}$. This is derived from the equations $J_0^2=-\mathrm{id}$, and $\omega(J_0 x, J_0y)=\omega(x,y)$. Now, for any two non-zero vectors $\xi, \eta \in \R^{2n}$ there is $Y \in T_{J_0} \mathcal{J}(\R^{2n})$ such that $Y\xi = \eta$. This is the content of \cite[Lemma 3.2.2]{MS12}; see \cite[Proposition 3.2.1]{MS12} for the application of this lemma in proving that the universal map is transverse to the zero section. Note that in \cite{MS12} they are working with the Cauchy-Riemann eqution which is not perturbed by $J_t X_{H_t}$, but the argument remains similar. 

Using the fact that the universal Floer map is transverse to the zero section, we will explain how to show that the space of almost complex structures for which the moduli space is a smooth manifold is of the second Baire category. Denote by $\mathcal{M}_{\mathcal{J}} (x_-, x_+)$ the universal moduli space $\mathcal{F}_H^{-1}(0)$. Consider the projection $\pi_2$ to the second coordinate $\mathcal{M}_{\mathcal{J}} (x_-, x_+) \ni (u,J) \mapsto J$. It follows from the Sard-Smale theorem that the space of regular values of $\pi_2$ is of the second Baire category. 

\begin{lemma}
Let $D: X \to Z$ be a Fredholm operator, and let $L:Y \to Z$ be a bounded operator such that $D\oplus L$ is surjective. Then the projection $\Pi: \ker D\oplus L \to Y$ to the second coordinate is Fredholm with $\ker \Pi \cong \ker D$ and $\mathrm{coker} \Pi \cong \mathrm{coker} D$.
\end{lemma}

This Lemma implies that for every regular value $J$ of $\pi_2$, operator $D_u$ is surjective, and hence, the moduli space $\mathcal{M}(x_-, x_+, H,J)$ is a smooth manifold. We conclude the section with the proof of this Lemma.

\begin{proof}
The kernel of the map $D\oplus L : X \times Z \to Y$ is given by $\ker(D \oplus L) = \{(x,z) \mid D x + L z=0\}$. Hence the kernel of $\Pi$ satisfies $$\ker \Pi = \{(x,0) \mid Dx = 0\} \cong \ker D.$$
On the other hand
$$
\frac{Z}{\mathrm{Im} \Pi} = \frac{Z}{L^{-1}(\mathrm{Im} D)}.
$$
Consider the map $\Bar{L}: Y \to \frac{\mathrm{Im} L }{\mathrm{Im} D \cap \mathrm{Im} L}$, defined by $\Bar{L} y := [L(y)]$.
The kernel of this map is $L^{-1}(\mathrm{Im} D)$, hence 
$$
\frac{Y}{L^{-1}(\mathrm{Im} D)} \cong \frac{\mathrm{Im} L }{\mathrm{Im} D \cap \mathrm{Im} L}.
$$
Since $D \oplus L$ is surjective we have
$$
\frac{Y}{\mathrm{Im} D} = \frac{\mathrm{Im} D + \mathrm{Im} L}{\mathrm{Im} D} \cong \frac{\mathrm{Im} L }{\mathrm{Im} D \cap \mathrm{Im} L}.
$$
This concludes the proof of $\mathrm{coker} \Pi \cong \mathrm{coker} D$.
\end{proof}

\subsubsection{Compactness}\label{sec:comp}
In this section, we are interested in compactness properties of the moduli space $\mathcal{M}(x_-, x_+)$. The foundational results about understanding the compactification of the space of $J$-holomorphic maps are due to Gromov \cite{Gr85}. The exposition in this section mostly follows \cite{Sa97}, and we cite \cite{MS12,AD14} to complement some details. Since we are dealing with a domain with boundary (infinite strip), compared to \cite{Sa97} where apart from broken trajectories, the holomorphic spheres could form if the energy is concentrated in the compact part of the cylinder, here we have a possibility of the disk that bubbles off when energy is concentrated on the boundary. Since our symplectic manifold and the Lagrangians are exact bubbling can not occur.

The \textit{energy} of a map $u:\R \times [0,1] \to M$ is
$$
E(u)=\int_{-\infty}^{\infty} \int_0^1 \|u_s\|^2_{J_t} dt ds,
$$
where $\|u_s\|^2_{J_t}:=\omega(\partial_s u, J_t \partial_s u)$. Since $u$ is the gradient trajectory of $\mathcal{A}_{H}$ we have the following 
\begin{lemma} \label{lemma:energy-difference}
	For $u \in \mathcal{M}(x_-, x_+,H,J_t)$ then $E(u)=\mathrm{A}_H(x_+) -\mathcal{A}_H(x_-)$.
\end{lemma}
As a consequence, the differential $d:CF_*(L_0, L_1;H_t) \to CF_{*-1}(L_0, L_1;H_t)$ drops the action (note that the input for the differential is at $+\infty$). The following lemma plays an important role in studying the compactness properties of moduli spaces. It guarantees that every solution of Floer's equation with finite energy is asymptotic to the chords of $H_t$ with endpoints on $L_i$.

\begin{lemma}\label{lemma:energy-orbit}
	If all the chords of $H_t$ are non-degenerate, and if $u:\R \times [0,1] \to M$ with $u(s,i) \in L_i$ solves equation (\ref{eq:floer_eq}) with $E(u) < \infty$ then there exist $y_{\pm} \in \mathrm{Crit} (\mathcal{A}_H)$ such that $u(s,t) \to y_{\pm}(t)$ when $s \to \pm \infty$, uniformly in $t$. 
\end{lemma}

One can additionally show that $u$ converges sub-exponentially. For the proof see \cite[Proposition 1.21]{Sa97} and \cite[\S 6.5]{AD14}. The proof uses Arzel\`a-Ascoli theorem to show that every map $y$ which is an approximate weak solution of $y'=X_{H_t}(y)$ must be $C^0$ close to a genuine solution. The second part is an elliptic estimate to show that if energy $E(u)$ is finite, then $|\partial_s u|$ converges to $0$ when $s \to \pm \infty$, uniformly in $t$.

\begin{prop}
	Let $u_n \in \mathcal{M}(x_-, x_+)$ be a sequence, assume that  $\|\partial_s u_n\|_{L^{\infty}}$ is not bounded. Then there exist a convergent sequence $(s'_n, t'_n)$, a sequence $R_n$, and a sequence $\epsilon_n$, with the following properties:
	\begin{itemize}
		\item $\epsilon_n R_n \to \infty$
		\item $w_n(s,t) := u_n(s'_n+s/R_n ,t'_n + t/R_n)$ converges to a $J_{t_{\infty}}$-holomorphic map $w:\Sigma \to M$ in $C^{\infty}_{loc}$ topology.
	\end{itemize} 
	The domain of $w_n$ is $\Sigma_n= \mathbb{D}(\epsilon_n R_n)$ if $t_{\infty} \notin \{0,1\}$, $\Sigma_n= \mathbb{D}(\epsilon_n R_n) \cap \mathbb{H}$ if $t_{\infty}=0$, and $\Sigma_n= \mathbb{D}(\epsilon_n R_n) \cap -\mathbb{H}$ if $t_{\infty}=1$.
\end{prop}
\begin{proof}
	The proof uses Hofer's lemma (see \cite[Lemma 6.6.3]{AD14}):
	\begin{lemma}
		Let $(X,d)$ be a complete metric space, and $f:X \to \R$ a non-negative continuous map. For every $x \in X$ with $f(x)>0$, and every $\delta>0$ there exists $x' \in X$, and $0<\epsilon\leq \delta$ such that 
		$$
		d(x,x')< 2 \epsilon, \hspace{2mm} \sup_{B(x', \epsilon)} f \leq 2 f(x'), \hspace{5mm} \epsilon f(x') \geq \delta f(x).
		$$
	\end{lemma}
	Since $\|\partial_s u_n\|_{L^{\infty}}$ is unbounded, after passing to a subsequence of $u_n$, there exists a sequence $(s_n,t_n) \in \R \times [0,1]$ such that $|\partial_s u_n(s_n, t_n)| \to \infty$. Since $u_n \in \mathcal{M}(x_-, x_+)$, the sequence $(s_n, t_n)$ stays in a compact set. Hence, there is a convergent subsequence (again denoted by $(s_n, t_n)$) $(s_n, t_n) \to (s_{\infty}, t_{\infty})$. 
	
	We apply Hofer's lemma to $X = \R \times [0,1]$, $f_n(s,t) = |\partial_s u_n(s,t)| $, $x_n =(s_n,t_n)$ and $\delta_n = |\partial_s u_n(s_n,t_n)|^{-1/2}$. We get $x'_n = (s'_n, t'_n)$ and $\epsilon_n \leq \delta_n$ that satisfy:
	\begin{itemize}
		\item $|(s_n - s'_n, t_n -t'_n)| < 2\epsilon_n \leq 2 \delta_n$,
		\item $\sup\limits_{\mathbb{D}((s'_n, t'_n), \epsilon_n ) \cap \R\times [0,1]} |\partial_s u_n| \leq 2 |\partial_s u_n(s'_n, t'_n)|$,
		\item $\epsilon_n |\partial_s u_n(s'_n, t'_n)| \geq \delta_n f_n(x) = |\partial_s u_n(s_n, t_n)|^{1/2}$.
	\end{itemize}
	 
	 Set $R_n = |\partial_s u_n(s'_n, t'_n)|$, it follows that $\epsilon_n R_n \to \infty$. If $t_{\infty}\notin \{0,1\}$, then for $n$ big enough $\mathbb{D}((s'_n, t'_n), \epsilon_n ) \subset \R\times (0,1)$. We have that
	$$
	\begin{aligned}
		\partial_s w_n(s,t)=\frac{\partial_s u_n((s'_n+s/R_n ,t'_n + t/R_n))}{R_n} \leq 2,
	\end{aligned}
	$$
	since $(s,t) \in \mathbb{D}((s'_n, t'_n), \epsilon_n ) $ and $\sup\limits_{\mathbb{D}((s'_n, t'_n), \epsilon_n ) \cap \R\times [0,1]} |\partial_s u_n| \leq 2 |\partial_s u_n(s'_n, t'_n)|$. Since the derivatives of $w_n$ are uniformly bounded, we get by Arzel\`a-Ascoli theorem that $w_n$ converges to $w$ in $C^0_{loc}$ (and hence in $C^{\infty}_{loc}$ by elliptic regularity). Furthermore, we have
	$$
	\begin{aligned}
	\partial_s w + J_{t_{\infty}}(w) \partial_t w &= \lim\limits_{n \to \infty} \partial_s(w_n) + J_t\partial_t (w_n)	\\
	&=\lim\limits_{n \to \infty}\frac{1}{R_n}(\partial_s u_n + J_t \partial_t u_n) = \lim\limits_{n \to \infty}\frac{1}{R_n} J_t X_{H}(u_n) = 0,
	\end{aligned}
	$$ 
	and
	$$
	\partial_s w(0,0) = \lim\limits_{n \to \infty} \partial_s w_n(0,0) = \lim\limits_{n \to \infty} \frac{\partial_s u_n(s'_n, t'_n)}{R_n} = 1.
	$$ 
	We have just shown that $w: \C \to M$ is a non-constant $J_{t_{\infty}}$-holomorphic map with bounded energy $E(w) \leq \mathcal{A}_H(x_+) - \mathcal{A}_H(x_-)$. By analogous arguments, one can analyze the case $t_{\infty} \in \{0,1\}$.
\end{proof}

By removal of singularities \cite[Theorem 4.1.2]{MS12} we get a non-constant $J_{t_{\infty}}$-holomorphic sphere $w:S^2 \to M$, or a disk $w:\mathbb{D} \to M$ with Lagrangian boundary conditions on $L_{t_{\infty}}$. Using that $J_{t_{\infty}}$ is $\omega$-compatible, we have:
$$
0<E(w) = \int_{S^2} w^*\omega =\int_{\emptyset} w^*\lambda = 0,
$$
or in the case of a disk
$$
0<E(w) = \int_{\mathbb{D}} w^*\omega = \int_{S^1} w^* \lambda =0,
$$
which leads to the contradiction. We conclude that the $s$ derivative is uniformly bounded for elements in $\mathcal{M}(x_-, x_+)$, which means that this is relatively compact space in the space of all continuous maps $u:\R \times [0,1] \to M$, with exponential decay towards $x_{\pm}$. Now, if a sequence $u_n \in \mathcal{M}(x_-, x_+)$ converges in $C^{0}_{loc}$ to an element which is not in $\mathcal{M}(x_-, x_+)$ we have that there is a sequence $s_n$ such that $u_n(s+s_n,t)$ converges to a solution of the Floer equation, which is not in $\mathcal{M}(x_-, x_+)$. Since $u$ has finite energy, by Lemma \ref{lemma:energy-orbit} there are orbits $y_{\pm}$ so that $u$ is asymptotic to $y_{\pm}$ when $s \to \pm \infty$. So we have proved
\begin{thm}\label{thm:compactness}
	 If $\mu(x_+) = \mu(x_-)+2$ then
	 $$
	 \partial \Bar{\mathcal{M}}(x_-, x_+) \subset \bigcup_{y} \Bar{\mathcal{M}}(x_-,y) \times \Bar{\mathcal{M}}(y, x_+).
	 $$
\end{thm}
One has a more general result. If all moduli spaces are cut-out transversally, then $\Bar{\mathcal{M}}(x_-, x_+)$ is a smooth manifold whose codimension $k$ stratum of the corner is a subset of $$\bigcup_{y_1, y_2,....,y_k} \Bar{\mathcal{M}}(x_-,y_1) \times \Bar{\mathcal{M}}(y_1,y_2)\times\cdots \times  \Bar{\mathcal{M}}(y_k, x_+).$$

In the following section on gluing, we demonstrate that every broken trajectory appears as a limit of genuine solutions, so the subset in Theorem \ref{thm:compactness} is, in fact, equality. 

\subsubsection{Gluing}\label{sec:gluing}
In this section, we show that for every broken trajectory $(u,v) \in \Bar{\mathcal{M}}(x_-, y) \times \Bar{\mathcal{M}}(y, x_+)$ there is a nearby solution $u \# v \in \Bar{\mathcal{M}}(x_-, x_+)$. We want to show 
\begin{thm}\label{thm:gluing}
If $\mu(x_+) - \mu(x_-) =2$, there exist $R_0 >0$ and an embedding
$$
\Psi: \Bar{\mathcal{M}}(x_-, y) \times \Bar{\mathcal{M}}(y, x_+) \times (R_0, \infty) \to \Bar{\mathcal{M}}(x_-, x_+),
$$
such that $\lim\limits_{R \to +\infty} \Psi(u,v, R) = (u,v)$.
\end{thm}

The idea is to pre-glue two solutions $(u,v)$ for each $R$  and obtain a smooth map $u \#_R v$, which satisfies Floer's equation up to an error that goes to $0$ when $R \to \infty$. Then by applying the implicit function theorem, we construct the unique solution which is close to the approximate solution, with the help of the kernel of the linearized operator $D_{u \#_R v}$ and a carefully chosen complement of $\ker D_{u \#_R v}$. 

Let $\beta:\R \to [0,1]$ be a smooth function which is $0$ for $s \leq 0$ and $1$ for $s\geq 1$. Furthermore, we assume that $0 \leq \beta' \leq 2$, and $-2 \leq \beta'' \leq 2$. Let $\eta(s,t), \xi(s,t)$ be two sections of $y^*TM$ such that $u(s,t) = \exp_{y(t)} (\eta(s,t))$ for $s \geq s_0$, and $v(s,t) = \exp_{y(t)} (\xi(s,t))$ for $s \leq - s_0$. For $R \geq 2 s_0 +1$ define
\begin{equation*}
u \#_R v (s,t) = \begin{cases} 
			u(s+R,t), &s\leq -R/2 - 1,\\
			\exp_{y(t)}(\beta(-s-R/2) \eta(s+R,t)), & -R/2 - 1 \leq s \leq -R/2,\\
			y(t), & -R/2 \leq s \leq R/2,\\
			\exp_{y(t)}(\beta(s - R/2) \xi(s-R,t)), & R/2 \leq s \leq R/2 +1,\\
			v(s-R, t), &s\geq R/2 +1.
			\end{cases}
\end{equation*}

In order to apply the implicit function theorem we need the following uniform surjectivity estimate. Set: $$D_R: W^{k+1,p}(\R \times [0,1], (u\#_Rv)^*TM) \to W^{k,p}(\R \times [0,1], (u\#_Rv)^*TM),$$ to be the linearization of the Floer section at $u\#_Rv$. The proof follows \cite[Proposition 3.9]{Sa97}.

\begin{prop}
There exist constants $c>0$ and $R_0>0$ such that, for every $R>R_0$ and every $\eta \in W^{2,p}(\R \times [0,1], (u\#_Rv)^*TM),$
$$
\|D^*_R \eta\|_{W^{1,p}} \leq c \|D_R D^*_R \eta\|_{L^p}
$$
\end{prop}
\begin{proof}
Set 
$$
u_R(s,t):=\begin{cases} u\#_Rv, &s \leq 0\\ y(t), &s\geq 0, \end{cases}  \hspace{5mm} v_R(s,t):=\begin{cases} y(t), &s\leq 0,\\ u\#_Rv, &s \geq 0.\end{cases} 			
$$
For $s \leq -R/2 - 1$ we have $u_R(s,t) = u(s+R,t)$, and similarly $v_R(s,t) = v(s-R,t)$ for $s \geq R/2 +1$. The subexponential convergence from Lemma \ref{lemma:energy-orbit} implies that the difference between $u_R(s,t)$ and $u(s+R, t)$ is exponentially small in $C^{\ell}$, for any $\ell$. An analogous statement holds for $v_R$ and the shift of $v$. Since the linearized operators $D_u$ and $D_v$ are surjective, we have for $R \geq R_0$ and $\eta_u \in W^{2,p}$:
$$
\| \eta_u\|_{W^{1,p}} \leq c_0 \|D_{u_R}^* \eta_u\|_{L^p}, \hspace{5mm} \| D^*_{u_R}\eta_u\|_{W^{1,p}} \leq c_1 \|D_{u_R} D_{u_R}^* \eta_u\|_{L^p},
$$ 
and the same inequalities with $u$ replaced with $v$ everywhere. 

Set $\beta_R(s) = \beta(s/R + 1/2)$, and for $\eta \in W^{1,p}(\R \times [0,1], (u\#_R v)^*TM)$ define:
$$
\begin{aligned}
\eta_u (s,t)&=(1-\beta_R(s,t))\eta(s,t) \in T_{u_R(s,t)} M,\\
\eta_v (s,t) &= \beta_R(s,t) \eta(s,t) \in T_{v_R(s,t)}M.
\end{aligned}
$$
Using the properties of $\beta$, and the fact that $D^*_R \eta_u = D^*_{u_R} \eta_u$ (and the same for $v_R$, and $\eta_v$) we have:
$$
\begin{aligned}
\|\eta\|_{W^{1,p}} &\leq \|\eta_u\|_{W^{1,p}} + \|\eta_v\|_{W^{1,p}}\\
&\leq c_0(\|D^*_{u_R} \eta_u\|_{L^p} + \|D^*_{v_R} \eta_v\|_{L^p})\\
&=c_0(\|D^*_{R} ((1-\beta_R) \eta)\|_{L^p} + \|D^*_{R} (\beta_R \eta)\|_{L^p}).
\end{aligned}
$$
Now $D^*_{R} (\beta_R \eta) = \beta_R D^*_R \eta - \beta'_R \eta$ implies:
$$
\|D^*_{R} (\beta_R \eta)\|_{L^p} + \|D^*_{R} ((1-\beta_R) \eta)\|_{L^p} \leq 2 \|D_R^* \eta\|_{L^p}+ \frac{4}{R} \|\eta\|_{L^p},
$$
using the fact that $0\leq \beta_R' \leq 2/R$. For $4c_0/R \leq 1/2$ we get: $$\|\eta\|_{W^{1,p}} \leq 4c_0 \|D^*_R \eta\|_{L^p}.$$
It follows from the definition of $\eta_u$ and $\eta_v$ that $D^*_R \eta = D^*_{u_R} \eta_u + D^*_{v_R} \eta_v$. From this we obtain:
$$
\begin{aligned}
\|D^*_R \eta\|_{W^{1,p}} &\leq \|D^*_{u_R} \eta_u\|_{W^{1,p}} + \|D^*_{v_R} \eta_v\|_{W^{1,p}} \\
&\leq c_1 (\|D_{u_R} D^*_{u_R} \eta_u\|_{L^{p}} + \|D_{v_R} D^*_{v_R} \eta_v\|_{L^{p}}) \\
& = c_1( \|D_R(\beta_R D^*_R \eta - \beta'_R \eta)\|_{L^p} + \|D_R((1-\beta_R)D^*_R \eta + \beta'_R \eta)\|_{L^p})\\
&\leq 2 c_1 \|D_R D^*_R \eta\|_{L^p} +\frac{4c_1}{R} \|D^*_R \eta - D_R \eta \|_{L^p} + \frac{2c_1}{R}\| \eta \|_{L^p}.
\end{aligned}
$$
Here we have used that $0 \leq \beta'_R \leq 2/R$, and that $-1/R\leq  \beta''_R \leq 1/R$ for $R>2$ (since $|\beta''_R| \leq 2/R^2$). Using that $\|D_R \eta\|_{L^p} \leq C \|\eta\|_{W^{1,p}}$ and $\|\eta\|_{L^p} \leq \|\eta \|_{W^{1,p}}$ we get $\|D^*_R \eta\|_{L^p} \leq 2 c_1 \|D_R D^*_R \eta\|_{L^p} +4c_1 /R  \|D^*_R \eta \|_{W^{1,p}} + c_2 / R \|\eta\|_{W^{1,p}}$. Together with $\|\eta\|_{W^{1,p}} \leq 4c_0 \|D^*_R \eta\|_{L^p}$ we get
$$
\|D^*_R \eta\|_{W^{1,p}} \leq 2c_1 \|D_R D^*_R \eta\|_{L^p} + \frac{2c_1 + 4c_0 c_2}{R}\|D^*_R \eta\|_{L^p},
$$
so for $(2c_1 + 4c_0 c_2)/R < 1/2$ we have $\|D^*_R \eta\|_{W^{1,p}} \leq 4c_1 \|D_R D^*_R \eta\|_{L^p}$.

\end{proof}

One can check that  $\mathcal{F}(u \#_R v)$ converges to $0$ as $R \to \infty$ in $L^p$ (and in $C^{\infty}$ topology). Since the vertical derivative $D_R$ is uniformly surjective, one can find a right inverse $G_R$ whose image is the complement of $\ker D_R$. Since $\mathcal{F}(u \#_R v)$ converges to zero, for $R$ large enough the image of $G_R$ will intersect $\exp_{u \#_R v}^{-1} (\mathcal{M}(x_-, x_+) )$ in the unique point, which determines the variation of $u \#_R v$ that makes it a genuine solution. For more details see \cite[Section 9.4]{AD14}. As a consequence, we get Theorem \ref{thm:gluing}.

\subsection{Continuation maps}

Given two admissible Hamiltonians $H^-_t$ and $H^+_t$, and two regular almost complex structures $J^-_t$ and $J^+_t$, we would like to have a comparison map 

$$
\Phi:CF_*(L_0, L_1; H^+_t, J^+_t) \to CF_*(L_0, L_1; H^-_t, J^-_t).
$$

In the Floer theory for closed Lagrangians, such a map always exists by counting elements in a zero-dimensional component of the moduli space that solves:

\begin{equation}\label{eq:continuation}
\partial_s u + J_{s,t}(\partial_t u - X_{H_{s,t}})= 0,
\end{equation}
where $J_{s,t}$ is a family of almost complex structures such that $J_{s,t} = J^{\pm}_t$ and $H_{s,t} = H^{\pm}_t$ for $\pm s \geq s_0$. If the path of almost complex structures $J_{s,t}$ is regular, we need to ensure that the zero-dimensional component is a finite set of points. For this, we again appeal to the maximum principle from Proposition \ref{prop:maximum}. This works for Hamiltonians $H^{\pm}_t$ that are linear at infinity, so $H^{\pm}_t(x,r)=a_{\pm} r + b_{\pm}$. From Proposition \ref{prop:maximum} we see that if  $a_- > a_+$ we can construct a homotopy $H_{s,t}$ such that $H_{s,t}(x,r) = f_s(r)$ outside of compact set, where $f_s$ satisfies $\partial_s \partial_r f_s(r) \leq 0$. Consider a cut-off function $\beta:\R \to [0,1]$, $\beta(s)=1$ for $s \geq s_0$, $\beta(s) = 0$ for $s \leq s_0$ and $\beta'(s) \geq 0$. The homotopy $$H_{s,t} = (1-\beta(s)) H^-_t + \beta(s) H^+_t,$$ satisfies the requirements for the maximum principle. Indeed, we have:

$$
\begin{aligned}
	\partial_s \partial_r H_{s,t} &= (1-\beta(s)) H^-_t + \beta(s) H^+_t)\\
	&= \partial_s( (1-\beta(s))a_- + \beta(s) a_+)\\
	&= \beta'(s) (a_+ - a_-) \leq 0.
\end{aligned}
$$

In general, we will use Lemma \ref{lemma:no_escape_2}, which can be applied to Hamiltonians that are contact at infinity, where it is enough to have $h^-_t\geq h^+_t$ for contact Hamiltonians $h_{\pm}:\partial M \to \R$.

Let $H^{\pm}_t$ be two admissible Hamiltonians, and let  $H_{s,t}$ and $J_{s,t}$ be regular homotopies. Hence, for every $x_{\pm} \in \mathrm{Crit}(\mathcal{A}_{H_{\pm}})$ we have that the moduli space $\mathcal{M}(x_-, x_+, H_{s,t}, J_{s,t})$ of solutions to (\ref{eq:continuation}) with the boundary on $L_i$, asymptotic to $x_{\pm}$ is a smooth manifold. 

The continuation map $\Phi_{H_{s,t}}:CF_*(L_0, L_1; H^+_t, J^+_t) \to CF_*(L_0, L_1; H^-_t, J^-_t)$ is defined on generators $x_+ \in CF_*(H^-_t, J^-_t)$ by:
$$
\Phi_{H_{s,t}} (x_+) = \sum_{\substack{x_-\\ \mu(x_-) = \mu(x_+)}} \#_2 \mathcal{M}(x_-, x_+, H_{s,t}, J_{s,t}) x_-,
$$
and extended to the whole chain complex by linearity. Note that here there is no $\R$ action by translation in the $s$-coordinate since the equation is not invariant under it.

\begin{lemma}\label{lemma:continuation}
The map $\Phi_{H_{s,t}}$ is a chain map.	
\end{lemma}
\begin{proof}
	The proof is similar to the proof of Lemma \ref{lemma:differential}. Here we count the elements in the boundary of the moduli space $\mathcal{M}(x_-, x_+, H_{s,t}, J_{s,t})$ where $\mu(x_+) - \mu(x_-) = 1$.
\end{proof}

Another important property is that the chain homotopy class of $\Phi_{H_{s,t}}$ does not depend on the choice of the homotopy between $H^-_t$ and $H^{+}_t$. The idea is to consider a homotopy of homotopies. Counting the elements in the associated zero-dimensional moduli space defines a chain homotopy. For details see \cite[Lemma 3.12]{Sa97}. As an outcome, we have a well-defined map:
$$
\Phi_{H^{\pm}}:HF_*(H^{+}_t, J^+_t) \to HF_*(H^{-}_t, J^-_t).
$$

\begin{prop}\label{prop:comp_cont}
	Let $H^1, H^2, H^3$ be three admissible Hamiltonians, such that there exist continuation maps: 
	$$
	\begin{aligned}
	\Phi_{H^1,H^2}&: HF_*(H^1, J^1) \to HF_*(H^2, J^2),\\
	\Phi_{H^2,H^3}&: HF_*(H^2, J^2) \to HF_*(H^3, J^3).
	\end{aligned}
	$$
	Then there exist a continuation map $\Phi_{H^1,H^3}: HF_*(H^1, J^1) \to HF_*(H^3, J^3)$, and $\Phi_{H^1,H^3} =\Phi_{H^2,H^3} \circ \Phi_{H^1,H^2}$. 
\end{prop}

The proof is similar to the proof that the chain homotopy class of $\Phi_{H_{s,t}}$ does not depend on the choice of homotopy $H_{s,t}$. A consequence of this Proposition is that the Floer homology group $HF_*(H, J)$ does not depend on the choice of an almost complex structure; however, it depends heavily on $H$, as we will see and exploit in the following chapters. This is not the case for the Floer homology of closed Lagrangians. In this case, the Floer homology does not depend on the choice of Hamiltonian as well.

The proof that $HF_*(H, J)$ does not depend on $J$ relies on the fact that in the case of the constant homotopies $H_{s,t} = H_t$ and $J_{s,t} = J_t$, the only elements are stationary solutions. This leads to the identity map on $HF_*(H, J)$. For this, we need the automatic transversality for stationary solutions $u(s,t) = x(t)$, for $x \in \mathrm{Crit}{\mathcal{A}_H}$. Additionally, we need that no non-constant solution can have the same asymptotic $x$. This is true because of the regularity of $J$. For the automatic transversality, see \cite[Propositon 3.7]{ASch06}

\subsubsection{Direct limit}
Morally speaking, the wrapped Floer homology is the Floer homology of a Hamiltonian that has an infinite slope. Here we will formalize this limiting procedure.

Let $I$ be an index set with a partial order $\preceq$. Let $\{V_{\alpha}\}_{\alpha \in I}$ be a collection of vector spaces over a field $\mathbb{K}$, together with homomorphism $h_{\alpha \beta}: V_{\alpha} \to V_\beta$ whenever $\alpha \preceq \beta$, such that 
$h_{\alpha \gamma} = h_{\beta \gamma} \circ h_{\alpha \beta}$ if $\alpha \preceq \beta \preceq \gamma$. 

\begin{defn}
	The direct limit of $\{V_{\alpha}, h_{\alpha \beta}\}$ is
	$$
	\lim_{\substack{\longrightarrow \\ I}} V_\alpha = \bigoplus V_\alpha / H,
	$$
	where $H$ is the subspace of $\bigoplus V_\alpha$ spanned by $i_{\beta} h_{\alpha \beta} (v_{\alpha})-i_{\alpha} v_{\alpha}$, for $\alpha \preceq \beta$ and $v_{\alpha} \in V_{\alpha}$ and $i_{\alpha}: V_{\alpha} \to \bigoplus V_\alpha$ is an obvious inclusion.
\end{defn} 

One way to formulate it is that two elements $a \in V_{\alpha}$ and $b \in V_{\beta}$ are equal in the direct limit if there is $\gamma$ such that $\alpha \preceq \gamma$, $\beta \preceq \gamma$ and $h_{\alpha \gamma} (a) = h_{\beta \gamma} (b)$. Note that from the definition there is the induced map $i_\alpha:V_{\alpha} \to 	\lim\limits_{\substack{\longrightarrow \\ I}} V_\alpha$.

A different perspective is to consider $(I, \preceq)$ as a category whose objects are elements of $I$, and the hom-sets $\mathrm{hom}(\alpha, \beta)$ consist of single points if $\alpha \preceq \beta$ and are the empty set otherwise. Then, the collection $\{V_{\alpha}, h_{\alpha \beta}\}$ is equivalent to a functor $V: I \to \mathrm{Vec}_{\mathbb{K}},$ from $I$ to the category of vector spaces $\mathrm{Vec}_{\mathbb{K}}$. Then, the direct limit is the colimit of the functor $V$.

\begin{ex}
	Consider $(\N, \leq)$ as the index set, and let $V_n = \Z_2$. Set $h_{nm} :V_n \to V_m$ to be the zero homomorphism, whenever $n \leq m$, then the direct limit is
	$$
	\lim_{\substack{\longrightarrow}} V_n = \bigoplus_{n \in \N} \Z_2.
	$$
\end{ex}

One important property of the direct limit is that it is an exact functor, meaning that it preserves short exact sequences. For more details on the direct limits of groups see \cite[\S 5]{Ro09}.

\begin{rem}
	Our chain complexes and their homologies are vector spaces over $\Z_2$. In terms of the definition of the direct limit, one can work in the category of groups (or modules) with essentially the same definition.
\end{rem}

A subset $J \subset I$ is called \textit{cofinal} if for every $\alpha \in I$, there is $\beta \in J$ such that $\alpha \preceq \beta$. One can show that:
$$
\lim_{\substack{\longrightarrow \\ I}} V_{\alpha} \cong \lim_{\substack{\longrightarrow \\ J}} V_{\alpha}.
$$

\subsubsection{Total wrapped Floer homology group}

The index set in our context is the set of all admissible Hamiltonians $\mathcal{H}$ for $L_0$ and $L_1$. For $H^{\pm}_t \in \mathcal{H}$ let $h^{\pm}_t : \partial M \to \R$ be corresponding contact Hamiltonians. Introduce the partial order $\preceq$ on $\mathcal{H}$ by
$$
H^+_t \preceq H^{-}_t \Leftrightarrow h^{+}_t(x) \leq h^{-}_t (x) \hspace{2mm} \forall x \in \partial M.
$$

\begin{defn}
	The wrapped Floer homology of the pair $L_0, L_1$ in $M$ is 
	$$
	HW_*(L_0, L_1; M) = \lim_{\substack{\longrightarrow \\ \mathcal{H}}} HF_*(L_0, L_1; H).
	$$
\end{defn}

If a Hamiltonian $H(x,r) = ar + b$ is linear at infinity is non-degenerate then $a$ is not a period of a Reeb chord from $\Lambda_0$ to $\Lambda_1$ with respect to the contact form $\alpha = \lambda \vert_{\partial M}$. Choosing an increasing, unbounded sequence $a_n \in \R$ such that $a_n$ is not in the \textit{spectrum} of $\alpha$:
$$\mathcal{A}(\Lambda_0, \Lambda_1, \alpha) := \left\{ \int \gamma^* \alpha \mid \gamma \text{ is a Reeb chord from } \Lambda_0 \text{ to } \Lambda_1 \right\},$$
 and picking non-degenerate Hamiltonians $H_n$ with slope $a_n$ we have that
 $$
 HW_*(L_0, L_1; M) \cong  \lim_{\longrightarrow } HF_*(L_0, L_1; H_n),
 $$
since $H_n$ is a cofinal family. If $L = L_0 = L_1$ we write $HW_*(L; M)$

\begin{ex}
Let $M = \R^{2n} \cong \C^n$ and $L = \R^n \times \{0\}$ take $H_{k}(z) = \frac{2k + 1}{4} \pi |z|^2$.	The calculation from Example \ref{ex:index} implies that the unique chord $x_k(t)=0$ is non-degenerate and has index $\mu (x_k) = nk$ so we have

$$
HF_*(\R^n, H_k) = \begin{cases}
	\Z_2, &*=nk\\
	0 &*\neq nk.
\end{cases}
$$
In fixed degree $*$ all groups $HF_*(\R^n, H_k)$ are zero for $k$ large enough, hence 
$$
HW_*(\R^n; \R^{2n}) = 0.
$$
\end{ex}
	
\subsubsection{Filtered wrapped Floer homology}
By Lemma \ref{lemma:energy-difference}, we have that the differential drops the action. Hence, the vector space $CF_*^{\leq a}(H)$ generated by the Hamiltonian chords $x$ with action $\mathcal{A}_H (x) \leq a$ is a chain subcomplex. Also, for $a < b$ we have that $CF_*^{\leq a}$ is a subcomplex of $CF_*^{\leq b}$. Denote by $CF_*^{(a, b]}$ the quotient complex $CF_*^{\leq a} / CF_*^{\leq b}$. In case that $b = +\infty$ we write $CF_*^{>a}$ for the quotient complex $CF_* / CF_*^{\leq a}$.

Since we have a short exact sequence of chain complexes:
$$
0 \to CF_*^{\leq a}(H) \to CF_*^{\leq b}(H) \to CF_*^{(a, b]}(H),
$$
this leads to the long exact sequence of homology groups:
$$
\cdots \to HF_*^{\leq a}(H) \to HF_*^{\leq b}(H) \to HF_*^{(a, b]}(H) \to HF_{*-1}^{\leq a}(H) \to \cdots.
$$
For solutions of the $s$-dependent Floer's equation, we have the following energy estimate
$$
\begin{aligned}
	0 \leq E(u) &= \int \omega (\partial_s u, \partial_t u - X_{H_s}) ds dt \\
	     &= \int u^* \omega - \int dH_s (\partial_s u) ds dt \\
	     &= \int u^* \omega - \int \partial_s (H_s(u(s,t)) ds dt + \int \partial_s H_s (u(s,t)) ds dt\\
	     &= \mathcal{A}_{H_+}(x_+) - \mathcal{A}_{H_-}(x_-) + \int \partial_s H_s (u(s,t)) ds dt.
\end{aligned}
$$
If $H_- \geq H_+$ everywhere, one can choose a homotopy $H_s$ from $H_-$ to $H_+$ such that $\partial_s H_s$ holds for all $x \in M$. Hence, the continuation map $\Phi_{H_s}$ induces a map of subcomplex $\Phi_{H_s}:CF_*^{\leq a}(H_+) \to CF_*^{\leq a}(H_-)$. Consequently we also have that $\Phi_{H_s}:CF_*^{(a,b]}(H_+) \to CF_*^{(a,b]}(H_-)$. The idea is to pass to the direct limit over a monotone cofinal family of Hamiltonians $\{H_i\}_{i \in \N}$.

The direct limit is an exact functor, meaning that if we have collections $\{A_{\alpha}\}_{\alpha \in I}$, $\{B_{\alpha}\}_{\alpha \in I}$, and $\{C_{\alpha}\}_{\alpha \in I}$ and collection of the short exact sequences:
$$
0 \to A_\alpha \overset{f_{\alpha}}{\to} B_{\alpha} \overset{g_{\alpha}}{\to} C_{\alpha} \to 0,
$$
then there are maps $f:= \lim\limits_{\longrightarrow} f_{\alpha}$ and $g:= \lim\limits_{\longrightarrow} g_{\alpha}$ such that:
$$
0 \to \lim\limits_{\longrightarrow} A_\alpha \overset{f}{\to} \lim\limits_{\longrightarrow} B_{\alpha} \overset{g}{\to} \lim\limits_{\longrightarrow} C_{\alpha} \to 0.
$$
The proof of this fact can be found in \cite[Proposition 5.33]{Ro09}. As an outcome, we get a long exact sequence:

	\begin{tikzcd}
		\cdots \arrow[r] & HW_*^{\leq a}(L_0, L_1; M) \arrow[r] &  HW_*^{\leq b}(L_0,L_1;M)\ar[overlay, out=0, in=180, looseness=2]{dl} \\
		 & HW_*^{(a, b]}(L_0, L_1;M) \arrow[r] & HW_{*-1}^{\leq a}(L_0, L_1; M)\arrow[r] & \cdots .\\
	\end{tikzcd}
	
The groups involved heavily depend on the choice of a monotone cofinal sequence $\{H_i\}_{i \in \N}$.

\section{Viterbo's transfer morphism}\label{sec:transf_morph}

In this section, we define the restriction map: $$\pi(M,W):HW_*(L_0, L_1; M) \to HW_*(\Bar{L}_0, \Bar{L}_1; W),$$ where $W \subset M$ is a codimension $0$ submanifold, whose boundary $\partial W$ is transverse to the Liouville vector field $X$ on $M$, and $\Bar{L}_i = L_i \cap W$. For example see Figure \ref{figure:subdomain}. The restriction map was introduced in the case of Symplectic (co)homology by Viterbo in \cite{Vi99}. In the wrapped setting, it was shown in \cite{AS10} that the restriction map also respects the $A_{\infty}$ algebra structure of the wrapped complex. One of the difficulties in showing that Viterbo's transfer morphisms respect higher products is finding the chain model for wrapped Floer (co)homology. In \cite[\S 3.7]{AS10}, they resolve this with a telescope construction.

\begin{figure}[h!]
	\centering
	\begin{tikzpicture}[scale = 0.7]
		\draw[black!40!red] (-1.4, 2.9) to [out = -75, in = 90] (-0.2, 0.85); 
		\draw[dashed, black!40!red](1.4, 3.1) to [out = -105, in = 90] (0.2, 1.15); 
		\draw[red]  (-0.2, 0.85) to [out = -90, in = 150] (0,-0.3);
		\draw[dashed, red] (0.2, 1.15) to [out = -90, in = 30] (0, -0.3); 
		\draw (-2,3) to[out = -75, in = 90] (-1,1);
		 \draw[blue] (-1,1) to[out=-90, in = 90] (-2,-1) to [out = 270, in = 180](0,-3);
		\draw (2,3) to[out = -105, in = 90] (1,1);
		\draw[blue] (1,1) to[out=-90, in = 90] (2,-1) to [out = 270, in = 0](0,-3);
		\draw (-0.8,3) .. controls (-0.5,1.5) and (0.5, 1.5) .. (0.8,3);
		\draw[dashed, blue] (-1,1) to [out = 15, in = 180] (0.2,1.15) to [out = 0, in = 165] (1,1);
		\draw[blue] (-1,1) to [out = -15, in = 180] (-0.2,0.85) to [out = 0, in = -165] (1,1); 
		\draw (-2,3) to [out = 15, in = 180] (-1.4, 3.1) to [out = 0, in = 165] (-0.8, 3);
		\draw (-2,3) to [out = -15, in = 180] (-1.4, 2.9) to [out = 0, in = -165] (-0.8, 3);
		\draw (0.8,3) to [out = 15, in = 180] (1.4, 3.1) to [out = 0, in = 165] (2, 3);
		\draw (0.8,3) to [out = -15, in = 180] (1.4, 2.9) to [out = 0, in = -165] (2, 3);
		\draw[blue] (-0.07, -0.23) to [out= -45, in = 135] (0,-0.3) to [out= -45, in = 45] (0,-1.7) to [out= -135, in=45] (-0.07, -1.77);
		\draw[blue] (0,-0.3) to [out= -135, in = 135] (0,-1.7);
		\draw[blue] node at (-1,-1) {$W$};
		\draw node at (2,1.5) {$M$};
		\draw[black!40!red] node at (-2, 1.5) {$L$};
		\draw[red] node at (0.7,0) {$\Bar{L}$};
	\end{tikzpicture}
	\caption{A Liouville subdomain $W \subset M$, and a Lagrangian $\Bar{L} = L \cap W$.}
	\label{figure:subdomain}
\end{figure}
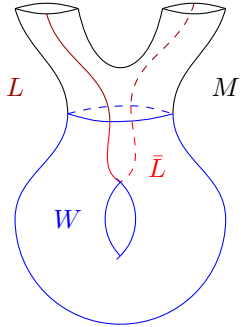
The following integrated maximum principle will be used in the definition of Viterbo's transfer map. This was introduced in \cite{AS10} for the wrapped Floer (co)homology and adapted to the setting of symplectic (co)homology in \cite{Ri13}. 

Consider two Hamiltonians $H_{\pm}$ that are contact at infinity, and assume that the corresponding contact Hamiltonians satisfy $h_-(x) \geq h_+(x) > 0$ for all $x \in \partial M$. Then, there exists a homotopy of Hamiltonians $H_s$ that satisfies 
\begin{itemize}
	\item $H_s(x,r) = H_{\pm}$ for $ \pm s \geq s_0$,
	\item $H_s(x,r) = h_s(x) r + b_s$ for $r \geq 1$,
	\item $\partial_s h_s (x) \leq 0$.
\end{itemize}
For each $s$ the positive contact Hamiltonian $h_s$ determines a new radial coordinate $r_s (p) = h_s(\pi_1(p)) r(p) $, where $\pi_1: \partial M \times (0, \infty) \to \partial M$ is the projection on the first coordinate. We will choose a homotopy of almost complex structures $J_s$ that satisfy
$\lambda \circ J = d r_s$. For $r_0 \geq \max h_s(x)$ set 
$$
Y_s = \left\{ \left(x, \frac{r_0}{h_s(x)}\right) \mid x \in \partial M \right\}.
$$ 
Let $u: \R \times [0,1] \to M$ be a solution of the continuation equation \ref{eq:continuation}, where $H_s$ and $J_s$ are as described in the previous paragraph. Let $S \subset \R \times [0,1]$ be a subdomain.
\begin{lemma}\label{lemma:no_escape_2}
If $u: S \to M$ satisfies $u(s,t) \in Y_s$ for every $(s,t) \in \partial S$, and $r_s \circ u(s,t) \geq r_0$ for all $(s,t) \in S$, then
$$
u(S) \subset \bigcup_{s \in \R} Y_s.
$$
\end{lemma}
\begin{proof}
	The energy of $u$ restricted to $S$ is given by
	$$
	\begin{aligned}
		0 \leq E(u) &= \int_S \omega(\partial_s u, J_s \partial_s u) ds dt = \int_S d \lambda(\partial_s u, \partial_t u - X_{H_s}) ds dt\\
		&= \int_S u^* d \lambda - \int_S dH_s(\partial_s u) ds dt \\
		& = \int_{\partial S} u^* \lambda - \int_S \partial_s (H_s (u)) ds dt + \int_S (\partial_s H_s) (u) ds dt \\
		&= \int_{\partial S} (u^* \lambda - H_s(u) dt) + \int_S (\partial_s H_s) (u) ds dt .
	\end{aligned}
	$$
	The equality in the third row follows from the partial integration, and the last equality uses Stoke's theorem. Note that $H_s(u(s,t)) = r_0 + b_s$ for $(s,t) \in \partial S$. Now we will calculate $\lambda(X_{H_s})$ along $Y_s$. Recall that $\lambda = r \alpha$ in the symplectization, and that $X_{H_s} = (X_{h_s}, - r dh_s(R_{\alpha}))$ (see Equation (\ref{eq:contact_at_inf})). Hence, on $Y_s$ following holds:
	$$
	\lambda(X_{H_s}) = \frac{r_0}{h_s(x)} \alpha(X_{h_s}) = \frac{r_0}{h_s(x)} h_s(x) = r_0. 
	$$
	Here, we have used that the radial coordinate $r$ restricted to $Y_s$ is given by $r_0 / h_s(x)$ and that for a contact Hamiltonian $h$, the equation that determines the Reeb component of the contact vector field $X_h$ is $\alpha(X_h) = h$. This gives us
	$$
	\begin{aligned}
		E(u)&= \int_{\partial S} (u^* \lambda - H_s(u) dt) + \int_S (\partial_s H_s) (u) ds dt\\
		&= \int_{\partial S} (u^* \lambda - \lambda(X_{H_s}) + \lambda (X_{H_s})- H_s(u) dt) + \int_S (\partial_s H_s) (u) ds dt\\
		&= \int_{\partial S} \lambda \circ (du - X_{H_s} \otimes dt) + \int_{\partial S} -b_s dt + \int_S (\partial_s H_s) (u) ds dt\\
		&= \int_{\partial S} \lambda \circ (du - X_{H_s} \otimes dt) + \int_S (\partial_s h_s) (\pi_1 \circ u) ds dt.\\
		& \leq \int_{\partial S} \lambda \circ (du - X_{H_s} \otimes dt).
	\end{aligned}
	$$
	Here we have used our calculation of $\lambda(X_{H_s})$, the Stoke's theorem for the $1$-form $-b_s dt$ and the assumption $\partial_s h_s \leq 0$.
	
	Now consider the splitting $\partial S = \partial_h S \cup \partial_v S$, where $\partial_h S$ is contained in the boundary of the strip, and $\partial_v S$ is the part of $\partial S$ in the interior of the strip. Both $dt$ and $\lambda$ vanish on $\partial_h S$ hence we get:
	$$
	E(u) \leq \int_{\partial_v S} \lambda \circ (du - X_{H_s} \otimes dt).
	$$
	The continuation equation is equivalent to: $$(du - X_{H_s} \otimes dt) \circ i = J_s \circ (du - X_{H_s} \otimes dt).$$ Together with $\lambda \circ J_s = d r_s$, this implies:
	$$
	E(u) \leq \int_{\partial_v S} -d r_s \circ (du - X_{H_s} \otimes dt) \circ i.
	$$
	Since $d r_s (X_{H_s}) = 0$ we get:
	$$
	E(u) \leq \int_{\partial_v S} -d r_s \circ du \circ i. 
	$$  
	If $n$ is the unit outward normal to $\partial_v S$, then the orientation of $\partial_v S$ is given by $i \cdot n$. Hence, we get that along $\partial_v S$ we are integrating $ d r_s ( du(n))$. Since $n$ points outward, this quantity is non-positive (since $r_s \circ u$ increases in the inward direction). So we get $E(u) \leq 0$. This implies $\partial_s u = 0$ along $S$, which together with $u(\partial S) \subset \bigcup Y_s$ implies $u(S) \subset \bigcup Y_s$.
\end{proof}

The Viterbo transfer morphism is defined as a composition of two maps. We will define a cofinal $H_n$ sequence on $M$ such that $H_n \vert_{W} <0$, and the quotient group $HW_*^{>0}(L_0, L_1; M)$ is isomorphic to $ HW_*(\Bar{L}_0, \Bar{L}_1; W).$ Then, $\pi(M,W)$ is the composition of the projection $\pi_*: HW_*(L_0,L_1;M) \to HW_*^{>0}(L_0, L_1; M),$ and the isomorphism $HW_*^{>0}(L_0, L_1; M) \cong HW_*(\Bar{L}_0, \Bar{L}_1; W)$. We follow \cite[\S 2.8]{Fa20} who was using the construction from \cite{Mc08}. Without loss of generality, we assume that $L_0$ and $L_1$ intersect transversally. This means that in the regions where we make our Hamiltonians constant, the only Hamiltonian chords are intersections between the two Lagrangians.

\begin{prop}
	There exists a cofinal sequence $H_n$ and a sequence of decreasing homotopies $H_{n,n+1}$ between them such that
	\begin{itemize}
		\item[(1)] $H_n \vert_{W}$, $H_{n,n+1} \vert_{W}$ are admissible Hamiltonians and homotopies for defining $HW_*(\Bar{L}_0, \Bar{L}_1; W)$,
		\item[(2)] all $1$-chords of $X_{H_n}$ from $L_0$ to $L_1$ in $W$ have positive action, and all the chords in $\widehat{M} \setminus W$ have negative action,
		\item[(3)] all solutions to the Floer (resp. continuation) equation of $H_n$ (resp. $H_{n,n+1}$) connecting $1$-chords in $W$ are entirely contained in $W$ for all admissible $J$ that are of contact type near $\partial W$.
	\end{itemize}
\end{prop}
\begin{proof}
Denote by $r_{W}$ and $r_{M}$ the radial coordinates on the completions $\widehat{W}$ and $\widehat{M}$. Using the Liouville flow of $X$ on $M$, one can symplectically embed $\widehat{W}$ to $\widehat{M}$. There exists a constant $C$ such that $\{ r_{W} \leq 1\} \subset \{r_{M} \leq C\}$, and consequently $\{r_{W} \leq r\} \subset \{r_{M} \leq Cr\}$. Set $\alpha_{W} = \lambda \vert_{\partial W}$, and assume that all Reeb chord between $\Bar{\Lambda}_0:= \Bar{L}_0 \cap \partial W$ and $\Bar{\Lambda}_1:= \Bar{L}_1 \cap \partial W$ are non-degenerate.

Choose an increasing unbounded sequence $a_n > 4C$ such that $a_n$ is not a period of a Reeb chord for $\alpha_W$ and $(4C)^{-1} a_n $ is not a period of a Reeb chord for $\alpha$. 

The idea is to construct $H_n$ that is radial near $r_{W} = 1$, and linear of slope $a_n$ on the region $r_{W} \leq r_n$, and of slope $(4C)^{-1} a_n$ for $r_{M} \geq 2C r_n$. Also, we want to control the actions of the chords that are created in the regions where the Hamiltonian is radial.

Set $$\delta_n:= d(a_n, \mathcal{A}(\Bar{\Lambda}_0, \Bar{\Lambda}_1 ; \alpha_{W})) = \min_{a \in \mathcal{A}(\Bar{\Lambda}_0, \Bar{\Lambda}_1 ; \alpha_{W})}|a_n - a| > 0,$$
where $\mathcal{A}(\Bar{\Lambda}_0, \Bar{\Lambda}_1 ; \alpha_{W})$ is the spectrum of periods of Reeb chords. Let $\epsilon_n>0$ be a decreasing sequence which converges to $0$, $\epsilon_1 < T_{\min}/(1+ T_{\min})$, where $ T_{\min}:=\min \mathcal{A}(\Bar{\Lambda}_0, \Bar{\Lambda}_1 ; \alpha_{W})$. Finally we choose an increasing sequence $r_n>\max\{2 + 2\epsilon_n/a_n, (a_n + \epsilon_n + \epsilon_n a_n) / \delta_n\}$.

Now, we define our Hamiltonian $H_n$:
\begin{itemize}
	\item on $W \setminus [1-\epsilon_n, 1] \times \partial W$, let $H_n \equiv -\epsilon_n$ be constant,
	\item on $[1-\epsilon_n, r_n] \times \partial W$, let $H_n(x, r_W) = f_n(r_W)$ where $f_n(1-\epsilon_n) = -\epsilon_n$, $0 \leq f_n' \leq a_n $, and $f_n(r_W) = a_n(r_W-1-\epsilon_n)$ for $1\leq r_W \leq r_n - \epsilon_n/a_{n} + \epsilon_n$,
	\item on $[r_n + \epsilon_n, 2r_n - \epsilon_n/C] \times \partial W$ we set $H_n=A_n$ for $A_n \in (a_n(r_n-1) - \epsilon_n, a_n(r_n - 1) )$, 
	\item on $r_M \leq 2C r_n - \epsilon_n$ we keep $H_n = A_n$,
	\item on $[2Cr_n - \epsilon_n, +\infty)$ we set $H_n (x,r_M) = g_n(r_M)$ where $g_n(2C r_n - \epsilon_n) = A_n$, $0 \leq g_n' \leq (4C)^{-1} a_n$, and $g_n(r_M)= (4C)^{-1} a_n(r_M - 2Cr_n) + a_n(r_n-1)$ for $r_M \geq 2Cr_n$.
\end{itemize}
\begin{figure}[h!]
	\centering
	\begin{tikzpicture}
		\draw[->, black!60!white] (-1.7,0.5) -- (6.2, 0.5);
		\draw[->, black!60!white] (-1.5,-0.2) -- (-1.5, 2.2);
		\draw[thick] (-1.5,0) -- (-1, 0) to [out = 0, in = 225] (0.5, 0.5) -- (1, 1) to [out=45, in = 180] (2,1.3) -- (3.5,1.3) to[out = 0, in=-150] (5, 1.5) -- (6,2);
		\draw[dashed, black!60!white] (2,1.3) -- (2,0.5) (5,1.5)--(5,0.5);
		\draw[black!60!white] (-1.6,1.3) -- (-1.4, 1.3) (-1.6, 0) -- (-1.4, 0) (0.5, 0.4) -- (0.5, 0.6) (2, 0.4) -- (2, 0.6) (5, 0.4) -- (5, 0.6);
		\draw node at (0.5,0) {$r_{W} = 1$};
		\draw node at (2,0) {$r_{W}=r_n$};
		\draw node at (5, 0) {$r_{M}= 2C r_n$};
		\draw node at (-1.9, 1.3) {$A_n$};
		\draw node at (-1.9, 0) {$\epsilon_n$};
	\end{tikzpicture}
	\label{figure:hamiltonian}
	\caption{The graph of $H_n$.}
\end{figure}
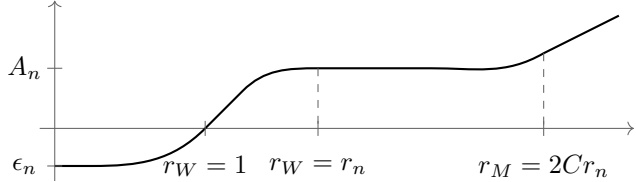
If a Hamiltonian is given by a function of radius $H=h(r)$ in a certain region, since the Hamiltonian vector field $X_H$ is equal to $h'(r) R_{\alpha}$, hence the action of a chord in this region is given by $r h'(r) - f(r)$, which equals to the minus of the $y$-coordinate of the intersection of the tangent line to $h$, and $y$-axis. Using this observation, the actions of the chords of our cofinal sequence $H_n$ are divided into five groups:
\begin{itemize}
\item[(a)] intersection points of $L_0$ and $L_1$ inside $W$ with action $\epsilon_n$,
\item[(b)] Hamiltonian chords that correspond to the non-constant Reeb chords near $r_{W}=1$ of action $r_{W} g'_n(r_{W}) - g_n(r_W)$ by construction this action is positive, and since $\epsilon_n$ converges to $0$, approximately equal to $g'_n(r_{W})$,
\item[(c)] non-constant chords $x$ near $r_{W} = r_n$ of action $r_{W}f'_n(r_W) - f_n{r_W}$, by the choice of $r_n$ and $A_n$ we have that these orbits have negative action:
$$
\begin{aligned}
	\mathcal{A}_{H_n}(x) &<r_n f'_n(r_{W}) - a_n(r_n-1) + \epsilon_n \\
	&< (a_n - \delta_n) (r_n + \epsilon_n) - a_n(r_n-1) + \epsilon_n  \\
	&= -\delta_n r_n +  a_n \epsilon_n - \delta_n \epsilon_n + a_n  +\epsilon_n\\
	&< -\delta_n r_n + a_n + a_n \epsilon_n +\epsilon_n <0.
\end{aligned}
$$
\item[(d)] intersection points of $L_0$ and $L_1$ in $\widehat M \setminus W$ of action $-A_n <0$,
\item[(e)] non-constant chords $x$ of $H_n$ near $r_{M} = 2 C r_n$ of action
$$
\begin{aligned}
	\mathcal{A}_{H_n}(x) &< g_n'(r_{M}) r_{M} - A_n < \frac{a_n}{4C} 2 r_n C- A_n  \\
	&< \frac{1}{2}a_n r_n - a_n(r_n-1) + \epsilon_n < \frac{1}{2}a_n(2 - r_n) + \epsilon_n < 0.
\end{aligned}
$$

This proves claim $(2)$ from the Proposition. For claim $(1)$ we need to carefully choose our functions $f_n$ so that $f_n \leq f_{n+1}$, which is possible since $\epsilon_{n+1} < \epsilon_n$, and $a_n < a_{n+1}$, so linear extensions of $H_n \vert_W$ form a cofinal sequence on $W$. It is not hard to check that $H_n \leq H_{n+1}$ holds globally. One easily checks that $A_{n+1} > A_{n}$. The only potential problem is that $H_n$ hits $H_{n+1}$ near $2r_{n+1} C$. There, we estimate:
$$
\begin{aligned}
	H_n(2r_{n+1}C) &= g_n(2r_{n+1}C)\\ 
	&= \frac{1}{4C} a_n (2r_{n+1} C- 2r_{n}C) + a_n(r_n-1)   \\
	& = \frac{1}{2}a_n r_{n+1} - \frac{1}{2} a_n r_n + a_n(r_n-1)\\
	& = \frac{1}{2} a_n r_{n+1} + a_n \left( \frac{r_n}{2} - 1\right) \\
	& < \frac{1}{2} a_{n+1} r_{n+1} + a_{n+1} \left( \frac{r_{n+1}}{2} - 1\right) \\
	& = a_{n+1}(r_{n+1} - 1) = H_{n+1}(2r_{n+1}C). 
\end{aligned}
$$
\end{itemize}
Since we have that $H_n \leq H_{n+1}$ globally, one can easily arrange monotone homotopies $H_{n, n+1}$. If $J$ is chosen to be admissible near $r_{W}=1$ by applying Lemma \ref{lemma:no_escape_2}, we get that all solutions of Floer's (continuation) equation that are joining chords that are in $W$ stay entirely in $W$.
\end{proof}

The quotient complex $CF_* ^{> 0}(L_0, L_1; M, H_n)$ is generated by the elements of positive action. By our construction, these are the orbits of $H_n \vert_{W}$ that are entirely contained in $W$, whose linear extension we use for the definition of $CF_{*}(\Bar{L}_0, \Bar{L}_1;W, H_n)$. Hence, we have an obvious map that sends generators $x \in CF_{*}(\Bar{L}_0, \Bar{L}_1;W, H_n)$ to its equivalence class $[x] \in CF_* ^{> 0}(L_0, L_1; M, H_n)$. This map is obviously a bijection by our construction. Also, by appealing to Lemma \ref{lemma:no_escape_2} again, we have that this map induces an isomorphism on homology. By passing to the direct limit, and again using the maximum principle for continuation strips, we get that:
$$
 \lim_{\longrightarrow } HF_*^{> 0}(L_0, L_1; M, H_n) \cong \lim_{\longrightarrow } HF_*(\Bar{L}_0, \Bar{L}_1; W, H_n).
$$

\section{Weinstein handles}\label{sec:handle_attach}

In this section, we sketch the construction of the contact surgery from \cite{We91}. For more details, we refer to \cite[\S 6]{Ge08}. 

Topologically, the outcome of contact surgery is the same as the outcome of topological surgery. First, we recall the idea of the topological surgery. Given an embedded sphere $S \cong S^{l}$ in a smooth manifold $Y^m$, together with a trivialization of the normal bundle of $S$, we can form a new manifold by removing $\nu S \cong S^l \times D^{m-l}$, and gluing back in $D^{l+1} \times S^{m-l-1}$ along the common boundary:
$$
Y' = Y \setminus \nu S \cup_{S^l \times S^{m-l-1}} D^{l+1} \times S^{m-l-1}.
$$
Note that it was important to fix a trivialization of the normal bundle, so we can use this map to identify $\partial (\nu S) \cong S^l \times S^{m-l-1}$. Also, the surgery comes with a cobordism $X$ between $Y$ and $Y'$ called the trace of the surgery obtained by gluing $D^{l+1} \times D^{m-l}$ to the trivial cobordism $Y \times [0,1]$:
$$
W = Y \times [0,1] \cup_{S^l \times D^{m-l}} D^{l+1} \times D^{m-l}.
$$

For a contact manifold $(Y, \xi )$, the outcome of the surgery $Y'$ carries a contact structure under some conditions, and the surgery provides a symplectic cobordism from $Y$ to $Y'$. Let $S\cong S^{k-1}$ be an isotropic sphere of a contact manifold. Contact structure $\xi$ carries a canonical conformal symplectic structure given by a positive multiple of $d\alpha$, where $\alpha$ is a $1$-form which determines $\xi$ together with its coorientation. The normal bundle $\nu S \cong T Y / TS$ splits into: $$\nu S \cong TY/ \xi  \oplus T^* S \oplus (TS)^{{d \alpha}} / TS,$$ where $(TS)^{d \alpha}$ denotes the symplectic orthogonal and $CSN(S):= (TS)^{{d \alpha}} / TS$ is the conformal symplectic normal bundle. $TY / \xi$ is naturally trivialized by the Reeb vector field. The stabilization $T S^{k-1}  \oplus \epsilon$ of $TS$ with a trivial line bundle $\epsilon$ carries a natural trivialization via inclusion $S^{k-1} \subset \R^{k}$. Consequently, the trivialization of $\nu S$ is equivalent to the trivialization of $CSN(S)$. 

Now, we will specify a local model, and after (symplectic) identification of a neighborhood of $S\times \{1\} \subset Y \times [0,1]$ with an open set of an isotropic sphere in a local model, we will be able to equip the outcome of the surgery with a contact structure. Since the local model is a subset of $\R^{2n}$ which has both positive and negative boundaries transverse to some Liouville vector field, the outcome of the surgery carries a contact structure as well.

To describe this more precisely, consider $\R^{2n}$ with the standard symplectic structure $\omega_{st} = \sum dx_i \wedge d y_i$ together with the following data:

\begin{equation*}
\begin{aligned}
		\lambda &= \frac{1}{2}\sum_{i=1}^k (3 x_j d y_j + y_j dx_j)  + \frac{1}{2} \sum_{i= k+1}^{n} ( x_j dy_j - y_j dx_j ),\\
		X &=  \frac{1}{2}\sum_{i=1}^k (3x_j \partial_{x_j} - y_j \partial_{y_j} ) + \frac{1}{2}\sum_{i=k+1}^n (x_j \partial_{x_j} + y_j \partial_{y_j}),\\
		\phi &=  \frac{1}{4}\sum_{i=1}^k \left(3x_i^2 - y_i^2\right) +\frac{1}{4} \sum_{i=k+1}^{n} \left(x_i^2 + y_i^2 \right).
\end{aligned}
\end{equation*}
Since $i_X \omega_{st} = \lambda$, and $ d \lambda = \omega_{st}$, $X$ is a Liouville vector field. We also have that $X = \nabla \phi$, so $X$ is transverse to $\Sigma_-:=\phi^{-1}(-1)$ because $\phi(0) = 0$ which implies further that $\Sigma_-$ is a contact hypersurface. It will be convenient to introduce the following notation:
$$
x= \frac{3}{4}\sum_{i=1}^k x_i^2, \hspace{5mm} y = \frac{1}{4} \sum_{i=1}^k  y_i^2, \hspace{5mm} z = \frac{1}{4} \sum_{i=k+1}^{n} \left(x_i^2 + y_i^2 \right).
$$
Let $S = \{ x = z = 0, y = 1\} \subset \Sigma_-$, since $\lambda \vert_{TS} =0$ it is an isotropic sphere in $\Sigma_-$. It follows from \cite[Proposition 4.2]{We91} that we can identify an open neighborhood of $U$ of $S \times \{1\} \subset Y \times (0,1] $  with an open neighborhood $U_-$ of $S \subset \{ \phi \leq -1\}$ which is matching the symplectic form $d (r \alpha)$ and the Liouville vector field $\partial_r$ on $U$ with $\omega_{st}$ and $X$ on $U_-$.

We want to find a contact hypersurface $\Sigma_+ \subset \R^{2n}$, which coincides with $\Sigma_-$ outside of $U_-$, and such that $\Sigma_-$ and $\Sigma_+$ bound a set diffeomorphic to $D^k \times D^{2n-k}$. This will be achieved by setting an appropriate function with level set $\Sigma_+$. Fix $\epsilon, \delta>0$ and pick a smooth function $g:\R \to \R$ with $0 \leq g' \leq (1+ 2\epsilon)^{-1}$ such that:
$$
g(t)=\begin{cases}
\frac{1}{1+2\epsilon} t,  &t\leq 1,\\
1, 	&t \geq 1+ 3\epsilon,
\end{cases}
$$
and set 
$$
\psi_{\delta}(x,y,z) = x - y + z - (1 + \epsilon) + (1+\epsilon) g (y + (x+z)/\delta).
$$
For $\delta$ small enough, this coincides with $\phi (x,y,z)= x-y+z$ outside of $U_-$, we also need to check that $\Sigma_{\delta} = \{ \psi_{\delta} = -1\}$ is transverse to $X$, so $\Sigma_{+}:=\Sigma_{\delta}$.

\begin{figure}[h!]
\centering
\begin{tikzpicture}
\draw node at (-4,3) [above] {$\Sigma_-$} ;
\draw[thick](-4,1) to [out=20, in=180] (-3,1.2) to  [out=0, in=160] (-2,1); 
\draw[thick](-4,3) to [out=-20, in=180] (-3,2.8) to [out=0, in=200] (-2,3);
\draw[->, black!50!white] (-3,1) -- (-3,3) ;
\draw[->, black!50!white] (-4,2) -- (-2,2);
\draw[black!80!white] node at (-3,3) [above right]{$\R^k$};
\draw[black!80!white] node at (-2,2) [above right] {$\R^{2n-k}$};
 \draw[red] (-1,1) to [out=20, in=270] (-0.5, 2) to [out=90, in=-20] (-1,3);
\draw[red] (1,1) to [out=160, in=270] (0.5, 2) to [out=90, in= 200]  (1,3); 

\draw[red] node at (1,2) {$\Sigma_+$};

\draw node at (3,2) {$H^{2n}_{k}$};
\draw[thick] (2,1) to [out=20, in=180] (3,1.2) to  [out=0, in=160] (4,1); 
\draw[thick](2,3) to [out=-20, in=180] (3,2.8) to [out=0, in=200] (4,3);
  \draw[red] (2,1) to [out=20, in=270] (2.5, 2) to [out=90, in=-20] (2,3);
 \draw[red] (4,1) to [out=160, in=270] (3.5, 2) to [out=90, in= 200]  (4,3); 
\end{tikzpicture}
\caption{A schematic picture of $\Sigma_-$, $\Sigma_+$ and $H_k^{2n},$ where $\R^{2n-k} = \{y=0\}$, and $\R^k = \{x=z=0\}$.}
\label{figure:handle}
\end{figure}
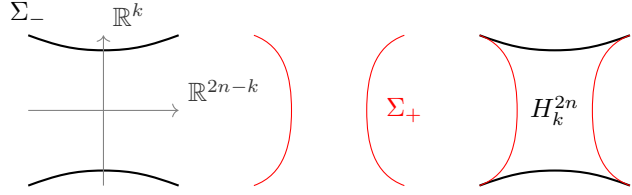

Denote by $X_x $, $X_y$, and $X_z$ Hamiltonian vector fields of functions $x, y$ and $z$. Since 
$$
d \psi_{\delta} = (1 + g'(1+\epsilon)/\delta ) dx + (-1 + g'(1+\epsilon) )dy + (1 + g'(1+\epsilon)/\delta ) dz,
$$
we get 
$$
X_{\psi_{\delta}} = (1 + (1+\epsilon)g'/\delta ) X_x +(-1 + (1+\epsilon) g')X_y + (1 + (1+\epsilon)g'/\delta) X_z.
$$
Using that
$$
X_x= \frac{3}{2}\sum_{i=1}^k x_i \partial_{y_i}, \hspace{5mm} X_y= -\frac{1}{2}\sum_{i=1}^k y_i \partial_{x_i}, \hspace{5mm} X_z= \frac{1}{2}\sum_{i=k+1}^n(x_i \partial_{y_i} -y_i \partial_{x_i}),
$$
leads to
$$
\begin{aligned}
d \psi_{\delta} (X) &= \omega_{st}(X, X_{\psi_{\delta}}) = \lambda(X_{\psi_{\delta}}) \\
&= (1 + (1+\epsilon)g'/\delta ) 3 x - (-1 + (1+\epsilon)g' ) y + (1 + (1+\epsilon)g'/\delta ) z.
\end{aligned}
$$
This is strictly positive, unless $x=y=z=0$, but $\psi_{\delta} (0,0,0) = -1 - \epsilon \neq -1$.

Define the handle by  $H_k^{2n} := \{\phi \geq -1\} \cap \{\psi_{\delta} \leq -1\}$ (see Figure \ref{figure:handle}). The outcome of the surgery is obtained by removing $U_- \cap Y \times \{1\}$ and gluing back in $\Sigma_{\delta} \setminus \Sigma_-$, or equivalently, the surgery is the right boundary of 
$$
Y\times [0,1] \cup_{S\times \{1\}} H_k^{2n}.
$$

We summarize this into the following
\begin{thm}
	Let $(Y^{2n-1}, \xi)$ be a contact manifold with a co-oriented contact structure $\xi$. Let $S \subset Y$ be an isotropic sphere with trivialized conformal symplectic normal bundle $CSN(S)$. There exist a contact structure on the surgery $$Y'=Y \setminus \nu S \cup_{S^{k-1} \times S^{2n-k-2}} D^k \times S^{2n-k-2},$$
	and a Liouville cobordism $(W, \lambda, X)$ such that $\partial W = Y \sqcup Y'$, and the Liouville vector field $X$ points inwards to $Y$ and outwards to $Y'$.
\end{thm}

In our situation, we also want to keep track of the Lagrangian $L_1 \subset M$, and how the surgery along $S\subset \Lambda_1 = \partial L_1$ affects it. Denote by:
$$H_k^n:= \{(x_1,...,x_n, y_1,...,y_n)\in H_k^{2n} \mid \forall i \hspace{2mm} x_i=0 \},$$
after a careful identification, we can assume that $\Lambda_1 \cap U_{-} \subset H_k^n \cap \Sigma_-$. Since $H_k^n$ is invariant under the Liouville flow and $\lambda\vert_{H_k^n} = 0$ we get
\begin{lemma}
	The Lagrangian $L_1 \cup_S H_k^n \subset (M \cup_{S} H_k^{2n}, \lambda, X)$ obtained by attaching $H_k^n$ to $L_1$ along $S$ satisfies $\lambda \vert_{L_1 \cup_S H_k^n} = 0$.
\end{lemma}
For more details see \cite[Lemma 2.6]{Ci02}. Here we have abused the notation for a Liouville form $\lambda$ and Liouville vector field $X$ using the identification of a neighborhood of $S$ in $\partial M$ and $S=\{x=z=0, y=1\} \subset \Sigma_-$.

\subsection{Construction of the cofinal sequence for the handle attachment}

In this section, we define a cofinal sequence of Hamiltonians on the surgered manifold $M \cup H_k^{2n}$, by extending a given cofinal sequence on $M$ that is linear near $\partial M$, to the completion of $M \cup H_k^{2n}$. The original idea in symplectic homology from \cite{Ci02} was to introduce only one constant periodic orbit $x(t) = 0 \in H_k^{2n}$ which has index tending to infinity. In \cite[Discussion 87]{Fa16-phd} and \cite[\S 3.3]{Fa20}, it was remarked that the original construction contains a small gap. Hamiltonians from \cite{Ci02} with listed properties can not be continuous after linear extension. Similar gap appears in \cite[p. 394]{Ir13}. 

Here, we follow the strategy of \cite{Fa20}, where it was allowed to create more periodic orbits in the region $\{x=y=0\}$. In our situation, this will create Hamiltonian $1$-chords with endpoints on $H_k^n$ which correspond to the Reeb chords on $\{x=y=0\} \cap \Sigma_{+}$ with endpoints on $H_k^n \cap \Sigma_{+}$. These chords are degenerate; however, their Robin-Salamon index goes to infinity. Hence, after a non-degenerate perturbation of our Hamiltonian near these chords, we will get clusters of Hamiltonian chords with Maslov index going to infinity, since the difference in index between a degenerate chord and its non-degenerate perturbation is uniformly bounded by $n/2$.

Also, in the process, we would like to make the handle thinner, since we would like to have that all the chords of $H_i$ are either entirely contained in $M \setminus U^i_-$, or in the handle $H_k^{2n}$. Here $U^i_-$ is a nested sequence of neighbourhoods of $S$ whose intersection is equal to $S$. This is possible since $S$ is isotropic of dimension $k < n$, i.e., sub-critical, hence, generically it avoids Reeb chords. As an outcome one can find a cofinal sequence $H_i$ admissible for defining $HW_*(L_0, L_1; M)$, and a sequence of handles (by shrinking $\delta$) such that if a chord of $H_i$ starts on $L_0$, then it needs larger period then $1$ to enter the region where $\partial M$ is identified with the handle $H_k^{2n}$ that is thin enough. This is avoided by the generic perturbation result, see \cite[Lemma 4.7, Lemma 5.4]{Ir13} or \cite[Appendix B]{Fa20}. See also \cite[Appendix B.1]{BCS24} for a similar transversality argument in the case of Riemannian metrics. 

The idea is to show that for the map $F_{\alpha}:M \times \R_+ \to M \times M$ given by $(x,t) \mapsto (x, \varphi_{\alpha}^t(x))$ can be made transverse to $\Lambda_0 \times S$. Since $\Lambda_0$ is Legendrian, and $S$ is isotropic, one can achieve transversality within the space of contact forms, and by the dimension argument we see that the image of $F_{\alpha}$ avoids $\Lambda_0 \times S$. Indeed the dimension of the domain is $2n$, the codimension of $\Lambda_0 \times S$ is $4n-2 - (n-1) - (k-1) = 3n - k > 2n$. As a consequence, we have the following

\begin{lemma}\label{lemma:generic_form}
	There exists a contact form $\alpha'$ on $\partial M$ such that for every $a>0$ there is $\delta > 0$ small enough so that every Reeb chord from $\Lambda_0$ to $\Lambda_1$ entering $H_\delta:= \{\phi \geq -1\} \cap \{ \psi_{\delta} \leq -1\}$ has period $>a$.
\end{lemma}

The proof is similar to \cite[Lemma 5.4]{Ir13}. We can also choose a contact form $\alpha'$ such that $Y' = \{(x, 1/h(x)) \mid x \in \partial M\} \subset \partial M \times (0, \infty)$ with the form $\lambda \vert_{Y'}$ is strictly contactomorphic to $(\partial M, \alpha')$, where the contact Hamiltonian $h: \partial M \to \R_+$ is bigger or equal then $1$. Additionally, we require that all the Reeb chords from $\Lambda_0$ to $\Lambda_1$ are non-degenerate. 

Without loss of generality, we will assume that our initial contact form $\alpha$ is as in Lemma \ref{lemma:generic_form}. Now, for fixed $\epsilon$, and fixed $\delta_0$, set $H_k^{2n} := \mathcal{H}_{\delta_0}$. Since $g$ is a monotone function, we have that $\psi_{\delta'} < \psi_{\delta}$ for $\delta < \delta'$ which further implies $\mathcal{H}_{\delta} \subset \mathcal{H}_{\delta'}$. Hence we can choose a nested family of handles $\mathcal{H}_{\delta_i}$, all of them contained in the initial handle $H_k^{2n}$, so that for each $a_i$ there is no Reeb chords from $\Lambda_0$ to $\Lambda_1$ of period $<a_i$ entering $\mathcal{H}_{\delta_i}$. All these handles give rise to the same completed Liouville manifold $\widehat{M \cup H_k^{2n}}$.

Note that the subspace $\{x=y=0\}$ is a Liouville subspace of the handle. The Liouville flow of $X$ is given by:
\begin{equation}\label{eq:liouville_flow_handle}
	\Phi_X^t( ... ,x_k,y_k, x_{k+1}, y_{k+1}, ...) = (..., e^{3t/2}x_k, e^{-t/2}y_k, e^{t/2}x_{k+1}, e^{t/2}y_{k+1}, ...).
\end{equation}
Hence $z$ is $1$-homogeneous with respect to $\Phi_X^{\log t}$ on $\{x=y=0\}$. For $z \leq \delta$ we have: 
$$
\psi_{\delta}= \left(1 + \frac{1+\epsilon}{\delta(1+2\epsilon)}\right)z - (1+\epsilon),
$$
let $r$ be a radial coordinate corresponding to $\Sigma_{+}= \{\psi_{\delta} = -1 \}$, i.e. $\Sigma_{+} = \{r=1\}$. In these coordinates, we have: 
$$
\psi_{\delta}(x,r) = a_{\delta} r - (1+ \epsilon).
$$
For $r=1$ we have $\psi_{\delta}=-1$, so we get $a_{\delta} - 1 - \epsilon = -1$, i.e., the slope $a_{\delta}$ is equal to $\epsilon$. 

The general idea is to use $\psi_{\delta}$ to extend a Hamiltonian on $M$ that is linear near $\partial M$. Before we continue, we will briefly explain the gap in \cite{Ir13}, which also appears in \cite{Ci02}. First of all, note that using the Lyapunov function $\mathcal{L}=\sum_{i=1}^k x_i y_i$ we know that there is no Hamiltonian chords of $\psi_{\delta}$ with endpoints on $H_k^n$ away from the critical set $Crit \mathcal{L} = \{x=y=0\}$ of $L$. Hence, all the potential chords are inside the Liouville subspace $\{x=y=0\}$. The gap that appears in \cite{Ci02} and \cite{Ir13} is that the Hamiltonians that satisfy all the requirements can not be continuous. In the notation of \cite{Ir13}, $H_i$ are Hamiltonians on $M$, and $G_i$ are the extensions to the handle. The requirements on $G_i$ on the handle are:
\begin{itemize}
	\item[1)]  $G_i = g_i(4z)$ for $g_i$ with $g_i' \notin (\pi/2)\Z$ on $\{x=y=0\}$,
	\item[2)] $\partial_{y_i} G_i <0$,
	\item[3)] there exists $A_i> m \pi/2$, $B_i>0$, $C_i<0$ such that
	$$
	G_i = G_i(0) + A_i 4z + B_i x + C_i y,
	$$
	\item[4)] $H_i = \alpha_i(r - \nu)$ for near $\partial M$, $G_i = \alpha_i(r - \nu)$ near $\partial (M \cup H_k^{2n})$. 
\end{itemize}
From $1)$ we get that $G_i(0) < \alpha_i(1 - \nu)$, and from $2)$ we get $G_i(0) > \alpha_i(1 - \nu)$ which is a contradiction. See \cite[Discussion 87]{Fa16-phd} for more details, where the gap was discovered.

The purpose of all these conditions is to rule out non-constant chords on the handle. The solution proposed in \cite{Fa16-phd, Fa20} is to allow non-constant periodic orbits, but to control their Conley-Zehnder index. We follow the same strategy in the setting of Lagrangians. We will first describe the procedure for the Hamiltonian that satisfies $H' = r-2$ near $\partial M$, and for a more general slope, we will take $aH' +2a + b$ that satisfies $ar+b$ near $\partial M$. This is possible since the handle remains unchanged if we replace $\phi$, and $\psi$ with $\phi' = a\phi+b$, and $\psi' =a \psi +b$, and $H_k^{2n} = \{ \phi' \geq -a + b \} \cap \{\psi' \leq -a +b\}$.

\subsubsection{Extension of $H'$ from $M$ to $M \cup H_k^{2n}$}
Let $\Sigma \subset \R^{2n}$ be a hypersurface transverse to $X$ with Reeb vector field of the form:
$$
R_{\alpha} = c_x X_x - c_y X_y + c_z X_z, 
$$
for smooth positive functions $c_x, c_y, c_z$, where $\alpha = \lambda \vert_{\Sigma}$. Consider a Hamiltonian $\tilde{h}_{\Sigma}(x,r) = a r + b$, and let $h_{\Sigma} =\tilde{h}_{\sigma} \circ \Phi^{-1} $ be its push-forward on $\Phi(\Sigma \times \R_+)$ by the Liouville flow $\Phi: \Sigma \times \R_+ \to \R^{2n}$, $(x,r) \mapsto \varphi_X^{\log r}(x)$. Then we have the following

\begin{lemma}\label{lemma:reeb}
	The Hamiltonian vector field of $X_{h_{\Sigma}}$ is given by
	$$
	X_{h_{\Sigma}} = C_x X_x - C_y X_y + C_z X_z,
	$$
	for smooth positive functions $C_x, C_y, C_z$.
\end{lemma}

\begin{proof}
	The Hamiltonian vector field $X_{\tilde{h}_{\Sigma}}$ of $\tilde{h}_{\Sigma}$ is equal to $a R_{\alpha}$, hence we have that the Hamiltonian vector field of $h_{\Sigma}$ restricted on $\varphi^{\log t}_X(\Sigma)$ is equal to $at R_t$ where $R_t$ is the Reeb vector field of $\varphi^{\log t}_X(\Sigma)$ with respect to the contact form $\lambda \vert_{\varphi^{\log t}_X(\Sigma)}$.
	
	The Reeb vector field $R_t$ satisfies $R_t = 1/t D \varphi^{\log t}_X R_{\alpha}$. This follows easily from ${\varphi^t_{X}}^* \lambda = e^t \lambda$  (for more details see \cite[Lemma 3.1]{Fa20}). From $R_{\alpha} = c_x X_x - c_y X_y + c_z X_z$, and from Equation (\ref{eq:liouville_flow_handle}) we get:
	$$
	X_{h_{\Sigma}} \vert_{\varphi^{\log t}_X(\Sigma)} = a D \varphi^{\log t}_X R_{\alpha} = t^{3/2}ac_x X_x - t^{-1/2}ac_y  X_y + t^{1/2}ac_z  X_z.
	$$
\end{proof}

Let $H'$ be an admissible Hamiltonian on $M$, which is of the form $r-2$ near $\partial M$. Since we have identified a neighborhood of $S \subset M$ with a neighborhood of $S$ in $\{ \phi \leq -1\}$, we will just consider the part $\partial M$ identified with a part of $\Sigma_- \subset \R^{2n}$. Consider a Hamiltonian $\tilde{h}_{\Sigma_-}: \Sigma_- \times \R_+$, $\tilde{h}_{\Sigma_-}(x,r) = r-2$ and   its push-forward  $h_{\Sigma_-}$ to $\R^{2n}$ by the Liouville flow. Under the identification of neighbourhood $U \subset \partial M \times (0,1]$ of $S$ with $U_- \subset \{ \phi \leq -1\}$ we have that $H'$ coincides with $h_{\Sigma_-}$. Since $\Sigma_- = \{\phi = -1\}$, we know that the Reeb vector field $R_{\lambda}$ is proportional to $X_{\phi} = X_x - X_y + X_z$. Along $S$ we have:
$$
\lambda(X_{\phi}) =  \frac{1}{2}\sum_{i=1}^k y_i dx_i (-X_y) = \frac{1}{4} \sum_{i=1}^k y_i^2 =y= 1.
$$
Equality $\lambda(X_{\phi})=1$ along $S$ implies $X_{h_{\Sigma_-}} = X_{\phi}$ along $S$, since $R_{\alpha} = X_{h_{\Sigma_-}}$ along $\Sigma_-$. We also have that $h_{\Sigma_-} \vert_S = \phi \vert_S = -1$. Since the Hamiltonian vector fields coincide, we can conclude that $dh_{\Sigma_-} = d\phi$ along $S$. As a consequence, given any neighborhood $U_{-}$ of $S$ there exists a function $\widehat{\phi}$ and a neighborhood $V_{-} \subset U_{-}$ such that $\widehat{\phi} = h_{\Sigma_-}$ on $\R^{2n} \setminus U_{-}$\footnote{The function $h_{\Sigma_-}$ is defined only on the image of $\Sigma_- \times \R_+$ by the Liouville flow of the, for simplicity we write $\R^{2n}$.}, $\widehat{\phi} = \phi$ on $V_{-}$ and $\widehat{\phi}$ is arbitrarily $C^1$-close to $h_{\Sigma_-}$. Since $X_{h_{\Sigma_-}} = C^{-}_x X_x - C^-_y X_y + C^-_z X_z$ by Lemma \ref{lemma:reeb}, we can achieve the same form for the Hamiltonian vector field of $\widehat{\phi}$. Additionally, since $X_x = X_z = 0$ along $\{x=z=0\}$, we can make $V_{-} \subset U_{-}$ arbitrarily thin in the $x$ and $z$ directions, and keeping the $y$ size fixed. Because of this, we can choose the same $\epsilon$ for all handles in the definition of $\psi_{\delta}$. Now, fix $\epsilon$ sufficiently small, and choose $\delta$ (depending on $U_-$) small enough so that $\mathcal{H}_{\delta} \cap \Sigma_- \subset V_{-}$. Define $\widehat{H}:M \cup \mathcal{H}_{\delta} \to \R$ by:

$$
\widehat{H} (p) = \begin{cases}
	\psi_{\delta} (p) & p \in (V_{-} \cap \{\phi \leq -1\}) \cup \mathcal{H}_{\delta},\\
	\widehat{\phi}(p) & p \in (U_{-} \cap \{\phi \leq -1\}) \setminus V_{-},\\
	H'(p) &p \in M \setminus U_-.
\end{cases}
$$
By construction, this extension is smooth. On $\mathcal{H}_{\delta} \cup U_{-}$ we have:
$$
X_{\widehat{H}} = \widehat{C}_x X_x - \widehat{C}_y X_y + \widehat{C}_z X_z.
$$
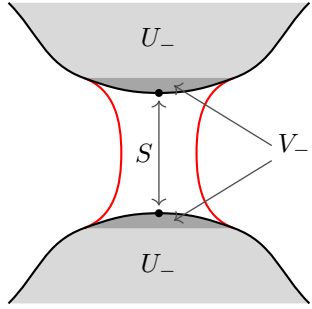
\begin{figure}[h!]
	\centering
	\begin{tikzpicture}
		\fill[black!15](-2,6)  to [out=45, in=200] (-1,7) to [out=20, in=180] (0,7.2) to  [out=0, in=160] (1,7) to [out=-20, in=135] (2,6); 
		\fill[black!15](-2,10)  to [out=-45, in=160] (-1,9) to [out=-20, in=180] (0,8.8) to [out=0, in=200] (1,9)to [out=20, in=225] (2,10);
		\fill[black!35] (-1,7) to [out=20, in=180] (0,7.2) to  [out=0, in=160] (1,7); 
		\fill[black!35](-1,9) to [out=-20, in=180] (0,8.8) to [out=0, in=200] (1,9);
		\draw[red, thick](-1,7) to [out=20, in=270] (-0.5,8) to [out=90, in=-20] (-1,9);
		\draw[red, thick](1,7) to [out=160, in=270] (0.5,8) to [out=90, in= 200]  (1,9);
		\draw[black, thick](-2,6)  to [out=45, in=200] (-1,7) to [out=20, in=180] (0,7.2) to  [out=0, in=160] (1,7) to [out=-20, in=135] (2,6); 
		\draw[black, thick](-2,10)  to [out=-45, in=160] (-1,9) to [out=-20, in=180] (0,8.8) to [out=0, in=200] (1,9)to [out=20, in=225] (2,10); 
		\node at (-0.2,8){$S$};
		\fill[black] (0,8.8) circle(1.5pt) (0,7.2) circle (1.5pt);
		\draw[<->, line width = 0.5pt, black!70] (0,7.3) -- (0,8.7);
		\node at (0,6.5) {$U_-$};
		\node at (0,9.5) {$U_-$};
		\draw[<-, black!70, line width = 0.5pt] (0.2,7.1) -- (1.5,7.9);
		\draw[<-, black!70, line width = 0.5pt] (0.2,8.9) -- (1.5,8.1);
		\node at (1.8,8.1) {$V_-$};
	\end{tikzpicture}
	\label{figure:neighborhoods_handle}
	\caption{Neighbourhoods $V_- \subset U_-$ of $S \subset M$}
\end{figure}

\subsubsection{Extension of $\widehat{H}$ to the symplectization}
The following interpolation Lemma is important for extending the Hamiltonian in the symplectization. As we have seen, $\psi_{\delta}$ is linear of slope $\epsilon$ on $\{x=y=0\}$, we would like to interpolate between $\psi_{\delta}$ and $r-2$ in the region $(\Sigma_+ \setminus \Sigma_- )\times \R_+$.  The following Lemma will give us control over the dynamics of the Hamiltonian of the interpolation, so that we can calculate the Maslov indices of the created chords on the handle $H_{k}^{2n}$.

Let $(\Sigma \times \R_+, \omega = d(r \alpha))$ be a symplectization of the closed contact manifold $(\Sigma, \alpha)$. For a $\omega$-compatible almost complex structure $J$ denote by $\| \cdot \|$ associated norm. 
\begin{lemma}\label{lemma:interpolation}
	Given $\epsilon, \delta, \rho >0$, there exists a smooth monotone increasing function $\beta:\R \to [0,1]$ such that $\beta(r) = 0$ for $r \leq 1-\epsilon$, $\beta(r)=1$ for $r\geq 1$, and for all smooth functions $\phi, \psi$ on $\Sigma \times \R_+$, with $\phi\vert_{\Sigma \times \{1\}} = \psi \vert_{\Sigma \times \{1\}}$ and $|\partial_r \phi - \partial_r \psi| < \rho$ for all $(x,r) \in \Sigma \times [1-\epsilon, 1]$ we have that Hamiltonian vector fields $X_{\phi}$, $X_{\psi}$ satisfy
	$$
	\sup \|X_{\phi + (\psi - \phi)\beta} -(X_{\phi} +(X_{\psi} - X_{\phi})\beta)\| \leq \delta.
	$$
\end{lemma}	

Put differently, for a careful choice of the interpolation function $\beta$, the Hamiltonian vector field of the interpolated function $\phi + (\psi - \phi)\beta$ does not differ by too much from the interpolation of Hamiltonian vector fields.

\begin{proof}
	The Hamiltonian vector field for $\phi + (\psi - \phi)\beta$ satisfies:
	$$
	X_{\phi + (\psi - \phi)\beta}= X_{\phi} + (X_{\psi} - X_{\phi}) \beta + (\psi-\phi) \beta' R, 
	$$
	so the difference $X_{\phi + (\psi - \phi)\beta} -(X_{\phi} +(X_{\psi} - X_{\phi})\beta)$ is equal to
	$(\psi-\phi) \beta' R$. Consequently we need to show that $\| (\psi-\phi) \beta' R\| \leq \delta$ for all $(x,r) \in \Sigma \times [1-\epsilon, 1]$. Using that $\phi$ and $\psi$ coincide on $\Sigma \times \{1\}$ we get
	$$
	\begin{aligned}
		\|(\psi - \phi)(x,r) \beta'(r) R(x)\| &= \| - \int_{r}^1 \partial_s (\psi - \phi) ds \beta' R(x) \|\\
		& \leq (1-r) \rho \beta'(r) \|R\|_{\infty}.
	\end{aligned}
	$$
	So, we will finish the proof if we can find $\beta$ so that $0\leq \beta'(r) \leq \delta (\rho \|R\| (r-1))^{-1}$ for $r \in [1-\epsilon, 1]$. Since $\int_{1- \epsilon}^1 (1-s)^{-1}ds = +\infty$ we can choose a smooth function $\tilde{\beta}$ such that $0 \leq \tilde{\beta} \leq \delta (\rho \|R\| (r-1))^{-1}$, $\tilde{\beta}$ vanishes for $s \notin (1-\epsilon, 1)$, and $\int_{1- \epsilon}^1 \tilde{\beta}(s) ds = 1$. Setting $\beta(r) = \int_{1-\epsilon}^r \tilde{\beta}(s)ds$ finishes the proof.
\end{proof}

We need to interpolate between $\psi_{\delta}$ and the push-forward $h_{\Sigma_{\delta}}$ of $\widehat{h}_{\Sigma_+}(x,r) = r-2$. For the index calculation, it will be important that the resulting Hamiltonian $H$ satisfies 

\begin{equation}\label{eq:form_of_ham}
	X_{H} = C_x X_x - C_y X_y + C_z X_z,
\end{equation}
for positive functions $C_x, C_y, C_z$. Let $\beta$ be a function from Lemma \ref{lemma:interpolation}, and set
$$
H:= \widehat{H} + (h_{\Sigma_{\delta}} - \widehat{H}) \beta(h_{\Sigma_{\delta}} +2).
$$
The region where the interpolation is taking effect is $(\Sigma_{\delta} \setminus U_-) \times \R_+$. By Lemma \ref{lemma:interpolation} we have that the Hamiltonian vector field $X_H$ is close to the interpolation $X_{\widehat{H}} + (X_{h_{\Sigma_{\delta}}} - X_{\widehat{H}}) \beta$. But since both Hamiltonian vector fields of $\widehat{H}$ and $h_{\Sigma_{\delta}}$ have desire form, we have that $X_H$ satisfies (\ref{eq:form_of_ham}).

Now, using the Lyapunov function $\mathcal{L} = \sum_{i=1}^k x_i y_i$ we see that all chords of $X_{H}$ with endpoints on $\widehat{L_1 \cup H_k^n}$ in $ \widehat{M \cup \mathcal{H}_{\delta}} \setminus \widehat{M \setminus V_-}$ must lie on $\{x=y=0\}$. Indeed, we have that $\mathcal{L}$ vanishes on $\widehat{L_1 \cup H_k^n}$, hence any chord $x$ of $X_H$ satisfies 
$$
\begin{aligned}
	\frac{d}{dt} \mathcal{L}(x(t)) &= d \mathcal{L} (X_H) \\
	&=  \sum_{i=1}^{k} y_i dx_i (- C_y X_y) + x_i dy_i (C_x X_x) \\
	&= \frac{1}{2}C_y y + \frac{1}{2}C_x  x.
\end{aligned}
$$
If $x$ has endpoints on $\widehat{L_1 \cup H_k^n}$, and $C_x, C_y >0$ we have that $x=y=0$, hence all the chords are contained in the Liouville subspace $\{x=y=0\}$, where we interpolate between two functions that only depend on the radial coordinate. 
\begin{rem}
	Note that if $L_0$ is not the same Lagrangian as $L_1$, the only $1$-chords of $H$ are the chords in $M$. This means that we can already establish the Invariance theorem in the case that $L_0$ is not affected by the surgery. We proceed with the case $L = L_0 = L_1$.
\end{rem}

\section{Invariance of wrapped Floer homology under subcritical handle attachment}\label{sec:inv_handle}

\subsection{Index calculation}

The goal of this section is to calculate the Robin-Salamon index of newly created orbits of $H$ in the region of the handle. On $\{x=y=0\}$ the Hamiltonian $H$ is an interpolation between $\psi_{\delta}$ and $h_{\Sigma_{+}}$ which in radial coordinates has the form $r-2$. Recall $\psi_{\delta}$ that has slope $\epsilon$ on $\{x=y=0\}$, hence in $C_z$ is a function depending only on $z$ that interpolates between $(1+(1+\epsilon) \delta^{-1} (1+2\epsilon)^{-1})$ and $\epsilon^{-1} (1+(1+\epsilon) \delta^{-1} (1+2\epsilon)^{-1}) $. Indeed, recall that: $$\psi_{\delta} = (1+(1+\epsilon) \delta^{-1}(1+2\epsilon)^{-1}) z - (1+\epsilon) = \epsilon r - (1+ \epsilon).$$
This holds for $z \leq \delta$. For $\epsilon$ small enough, independent of $\delta$, we get that this holds for $\{r \leq 1\} \subset \{ z \leq \delta\}$ on $\{x=y=0\}$. Since the elements of the cofinal family will be of the form $ aH + 2a +b$, we need to calculate the indices for scaled Hamiltonians. The Hamiltonian vector field of $ aH + 2a +b$ is given by $ aX_{H}$, and denote by $\phi_{aH}^t$ its flow. Since $X_x = X_y = 0$ along $\{x=y=0\}$, there we have $aX_{H} = aC_z X_z$, hence the $z$ coordinate is fixed along the flow of $\phi_{a H}^t$ and consequently $C_z$ is constant along the flow. Because of this, we can explicitly calculate the flow:
$$
\phi_{a H}^t (0,...,0, z_{k+1}, ..., z_{n}) = (0,...,0, e^{i a C_z t/2}z_{k+1}, ..., e^{i a C_z t/2}z_{n}),
$$
since $X_z = (0,...,0, i z_{k+1} /2, ..., i z_{n}/2)$. As a consequence we get that the $1$-chords on $\{ (iy_1, ..., i y_n) \mid y_i \in \R\}$ are the constant orbit at $0$ and $S^{n-k-1}$ family of $1$-chords $\gamma$ starting and ending on $\{ (0,..., 0, i y_{k+1}, ..., i y_{n}) \mid \sum y_i ^2 = c >0\}$ that appears on the level $z$ for which $a C_z /2 \in \pi \Z$. Since we will use the splitting axiom, we recall the definition of the Robbin-Salamon index in $\R^2$ from \cite{RS93}. For a Lagrangian path $\Lambda(t) = e^{i a(t)} V$ with respect to the fixed vertical Lagrangian $V = i \R$:
$$
\mu_{RS} (\Lambda(t), V) = \frac{1}{2} \dim ( e^{i a(0)} V \cap V) +  ie^{i a(t)} \vert_{(0,1)} \cap V +  \frac{1}{2} \dim ( e^{i a(1)} V \cap V),
$$ 
where for a path $\alpha:(0,1) \to \R^2$, $\alpha \cap V$ is the signed intersection number. Our path of Lagrangians will be determined by the push-forward of $i \R^n$ by $ D\phi_{aH}^t$. Let $x$ be a $1$-chord of $X_{aH}$ with endpoints on $i \R^n$. We need to solve the ODE given by $\frac{d}{dt}D \phi_{aH}^t (x(0)) = aD  X_{H} (\phi_{aH}^t (x(0))$. Since $X_{aH} = aC_x X_x - a C_y X_y + a C_z X_z$, and $X_x= X_y = 0$ along our curves, derivatives of $C_x$ and $C_y$ do not contribute. Hence, in $(x_i, y_i)$ plane for $i \leq k$ we only have contributions of $C_xD X_x - C_y D X_y$. In the plane $(x_i, y_i)$ we have $X_x(x_j,y_j) = 3i x/2 $ and $X_y (x_j, y_j) = -y_j/2$. Hence $$C_xD X_x - C_y D X_y = \frac{1}{2}\sum_{i=1}^{k}(C_y \partial_{x_i} dy_i + 3 C_x  \partial_{y_i} dx_i),$$ or, equivalently, the $i$-th diagonal $2\times 2$-component $\Phi_i$ of $D \phi_{aH}^t (x(0))$ satisfies:
$$
\frac{d}{dt} \Phi_i(t) = a\begin{bmatrix}
	0 & C_y/2\\
	3C_x/2 & 0
\end{bmatrix} \Phi_i(t).
$$
On the other hand $D C_z$ has parts involving $dx_i, dy_i$ for $i \leq k$, but $X_z$ does not depend on $x_i, y_i$ for $i \leq k$, also, along $\{x=y=0\}$ function $C_z$ depends only on $z$ hence this part of $DC_z$ is of the form  $C'(z)dz$, but $C_z$ is constant along the flow, so $DC_z$ does not contribute. The component of $X_z$ in $(x_i, y_i)$-plane for $i \geq k+1$ satisfies $D X_z =i/2 \cdot \mathrm{Id}$, so for $i \geq k+1$ we get
$$
\frac{d}{dt} \Phi_i(t) = a i C_z /2 \Phi_i(t).
$$
Since $\Phi(t) = D \phi_{aH}^t (x(0)) = \mathrm{diag}(\Phi_1(t),...,\Phi_n(t))$, we can calculate the Robin-Salamon indices using the splitting axiom:
$$
\mu_{RS}(x) = \sum_{i=1}^{n} \mu_{RS}(\Phi_i(t) V, V).
$$ 
For $i  \leq k$, we have that:
$$
\begin{bmatrix}
	0 & C_y/2\\
	3C_x/2 & 0
\end{bmatrix}  =  \begin{bmatrix}
0 & -1\\
1 & 0
\end{bmatrix}  \cdot \begin{bmatrix}
3C_x /2 & 0\\
0 & -C_y/2
\end{bmatrix} ,
$$
since the symmetric matrix has signature $0$, we have that $\mu_{RS} (\Phi_i(t) V, V) = 1/2$ for all $i \leq k$. To see this differently, let's look at the path $$\Phi_i(t) \begin{bmatrix}
	0\\
	1
\end{bmatrix} = \begin{bmatrix}
\Phi^{12}_i(t)\\
\Phi^{22}_i(t)
\end{bmatrix}.$$ This path satisfies the system of ODE:
$$
\begin{aligned}
	{\Phi^{12}_i}'(t) &= a(t) \Phi^{22}_i(t)\\
	{\Phi^{22}_i}'(t) &= b(t) \Phi^{12}_i(t),
\end{aligned}$$ with initial the conditions $\Phi^{12}_i(0) = 0$, and $\Phi^{22}_i(0)=1$. Since $a,b >0$ we can easily see that $\Phi^{22}(t)>1$ for $t > 0$, hence the intersection number between the path $(\Phi^{12}_i(t), \Phi^{22}_i(t))$ and $V$ on $(0,1)$ is $0$, hence $\mu_{RS} (\Phi_i(t) V, V) = 1/2$. 

Now, for $i \geq k+1$ we have that $\Phi_i(t) = e^{i a C_z t/2}$, and $a C_z \in 2 \pi \Z$. Since this is a loop of Lagrangians, we have $\dim (\Phi(1) V \cap V) = 1$, and the Robin-Salamon index satisfies
$$
\mu_{RS}(\Phi_i(t) V, V) = \frac{1}{2} + \frac{aC_z}{2\pi} - 1 + \frac{1}{2}.
$$
In total, we get
$$
\mu_{RS}(x) = \frac{n}{2} + (n-k) \left(\frac{aC_z}{2\pi} - \frac{1}{2}\right).
$$

Now, after resolving degeneracy, this Morse-Bott family will split into two clusters of non-degenerate orbits whose Maslov indices satisfy $$\begin{aligned}
	\mu(x_-) &\in \left(\frac{(n-k) a C_z}{2 \pi} -n + \frac{k}{2},  \frac{(n-k) a C_z}{2 \pi} +\frac{k}{2}\right)\\ 
	\mu(x_+) &\in \left(\frac{(n-k) a C_z}{2 \pi} - \frac{k}{2} -1,  \frac{(n-k) a C_z}{2 \pi} +n -\frac{k}{2}-1\right).
\end{aligned}$$
Here, we were using the relationship $\mu(x) = \mu_{RS}(x) - n/2$ together with the following two facts: given an isolated degenerate orbit $x$, after a perturbation created non-degenerate orbits $x'$ satisfy $$\mu_{RS}(x') \in \left(\mu_{RS}(x) - \frac{n}{2}, \mu_{RS} + \frac{n}{2}\right).$$
On the other hand, if a degenerate orbits form a manifold, then we will create clusters of non-degenerate orbits corresponding to each critical point of a Morse function on a given manifold, and the index will be shifted by the Morse index of a critical point. For more details, see \cite[Proposition 2.2]{CFHW96}. Since $C_z$ is bounded from below, we see that the Maslov indices of created orbits go to infinity as $a \to +\infty$.
\subsection{The proof of the invariance}

For the invariance theorem, we could appeal to Viterbo's restriction morphism; however, we would need to be careful about the actions of the newly created orbits. We do not have on the nose that $HW_*^{\geq 0}(L_0, L_1 \cup H_k^n; \widehat{M \cup H_k^{2n}}) \cong HW_*(L_0, L_1; M)$ by the action reasons. We could play the index argument, but the same index argument would show that $HW_*^{\geq 0}(L_0, L_1 \cup H_k^n; \widehat{M \cup H_k^{2n}}) \cong HW_*(L_0, L_1 \cup H_k^n; \widehat{M \cup H_k^{2n}})$. Hence, we will directly show:
$$
HW_*(L_0, L_1; M) \cong HW_*(L_0, L_1 \cup H_k^n; \widehat{M \cup H_k^{2n}}).
$$ 

Choose a cofinal family $H_i = a_i H_{\delta_i} +2a_i +b_i$\footnote{note that the definition of the initial Hamiltonian $H$ of that satisfies $r-2$ in the symplectization depends on $\delta$.} where $\delta_i$ is a decreasing sequence such that for each $a_i$ all Reeb chords on $\partial M$ with periods smaller than $a_i$ do not enter the handle $\mathcal{H}_{\delta_i}$. Since for $\delta_{i+1} \leq \delta_i$ we have that $\Sigma_{\delta_{i+1}} \subset \mathcal{H}_{\delta}$, i.e. the radial coordinate $r_{\delta_{i}}$ corresponding to $\Sigma_{\delta_{i}}$ satisfies $$r_{\delta_{i}} \vert_{\Sigma_{\delta_{i+1}}} \leq 1.$$ Consequently we have $a_{i} r_{\delta_{i}} \leq a_{i+1} r_{\delta_{i}} \leq a_{i+1} r_{\delta_{i+1}}$. 

We have that linear extensions of $H_i \vert_{M \setminus (1-\delta_{i},1] \times \partial M}$ calculate the wrapped Floer homology $HW_*(L_0,L_1;M)$. On the other hand, the same cofinal family calculates $HW_*(L_0, L_1 \cup H_k^n; \widehat{M \cup H_k^{2n}}).$

Fix a degree $*=k$. There exists $c(k) \in \Z$ such that for $i \geq c(k)$ we have that the chain complexes $CF_k(L_0, L_1, H_i \vert M ; M)$ and $CF_k(L_0,  L_1 \cup H_k^n, H_i; \widehat{M \cup H_k^{2n}})$ have the same generators. Choose a sequence of almost complex structures $J_i$ that is of contact type near $\partial M \times \{1-\delta_i\}$, and of contact type near $\partial (M\cup H_k^ {2n})$. By Lemma \ref{lemma:no_escape_2}, we have that all solutions of Floer's equation and the continuation equation that connect orbits that are in $M$, stay entirely in $M$. For $i \geq c(k)$ we have the following commutative diagram:
 \begin{equation*}
 	\centering
 	\begin{tikzcd}
 		HF_k(L_0, L_1, H_i \vert M) \arrow[r, "\Psi_i"]\arrow[d, "\Phi_{H_{i,i+1}\vert_M}"] &  HF_k(L_0,L_1 \cup H_k^n, H_i)\arrow[d, "\Phi_{H_{i,i+1}}"] \\
 		HF_k(L_0, L_1, H_{i+1} \vert_M) \arrow[r, "\Psi_{i+1}"] & HW_k(L_0, L_1\cup H_k^n, H_i),
 	\end{tikzcd}
 \end{equation*}
 where $\Psi_i$ is the isomorphism which acts as identity on generators, $\Phi_{H_{i,i+1}\vert_M}$ and $\Phi_{H_{i,i+1}}$ are continuation maps. Passing to the direct limit, we finish the proof.

{\footnotesize
	\bibliography{citations}

\begin{thebibliography}{CFHW96}

\bibitem[Abo15]{Ab15}
M.~Abouzaid.
\newblock Symplectic cohomology and {V}iterbo's theorem.
\newblock In {\em Free Loop Spaces in Geometry and Topology}, pages 271--486.
  European Mathematical Society, 2015.

\bibitem[AD14]{AD14}
M.~Audin and M.~Damian.
\newblock {\em Morse Theory and Floer Homology}, volume 223 of {\em
  Universitext}.
\newblock Springer, 2014.
\newblock Translted by Reinie Ern{\'e}.

\bibitem[AFKP12]{AFKP12}
P.~Albers, U.~Frauenfelder, O.~V. Koert, and G.~P. Paternain.
\newblock Contact geometry of the restricted three-body problem.
\newblock {\em Communications on Pure and Applied Mathematics}, 65(2):229--263,
  2012.

\bibitem[APS08]{APS08}
A.~Abbondandolo, A.~Portaluri, and M.~Schwarz.
\newblock The homology of path spaces and {F}loer homology with conormal
  boundary conditions.
\newblock {\em J. Fixed Point Theory Appl.}, 4:263--293, 2008.

\bibitem[Arn86]{Ar86}
V.~I. Arnol'd.
\newblock {First steps in symplectic topology}.
\newblock {\em Russ. Math. Surv.}, 41(6):--21, 1986.

\bibitem[AS06]{ASch06}
A.~Abbondandolo and M.~Schwarz.
\newblock On the {F}loer homology of cotangent bundles.
\newblock {\em Comm. Pure Appl. Math.}, 59:254--316, 2006.

\bibitem[AS10a]{ASch10}
A.~Abbondandolo and M.~Schwarz.
\newblock Floer homology of cotangent bundles and the loop product.
\newblock {\em Geom. Topol.}, 14:1569--1722, 2010.

\bibitem[AS10b]{AS10}
M.~Abouzaid and P.~Seidel.
\newblock An open string analogue of {V}iterbo functoriality.
\newblock {\em Geom. Topol.}, 14:627--718, 2010.

\bibitem[Aur14]{Au14}
D.~Auroux.
\newblock {\em {A Beginner's Introduction to Fukaya Categories}}, pages
  85--136.
\newblock Springer International Publishing, 2014.

\bibitem[BC24]{BC24}
F.~Bro{\'c}i{\'c} and D.~Cant.
\newblock Bordism classes of loops and floer's equation in cotangent bundles.
\newblock {\em Journal of Fixed Point Theory and Applications}, 26(3):26, 2024.

\bibitem[BCS25]{BCS24}
F.~Bro{\'c}i{\'c}, D.~Cant, and E.~Shelukhin.
\newblock The chord conjecture for conormal bundles.
\newblock {\em Mathematische Annalen}, 2025.

\bibitem[BF25]{BF25}
F.~Bro\'ci\'c and U.~Frauenfelder.
\newblock {Wrapped Floer homology and the circular restricted three body
  problem}.
\newblock draft, 2025.

\bibitem[CFHW96]{CFHW96}
K.~Cieliebak, A.~Floer, H.~Hofer, and K.~Wysocki.
\newblock Applications of symplectic homology ii: Stability of the action
  spectrum.
\newblock {\em Mathematische Zeitschrift}, 223(1):27--45, 1996.

\bibitem[Cie02]{Ci02}
K.~Cieliebak.
\newblock Handle attaching in symplectic homology and the chord conjecture.
\newblock {\em J. European Math. Soc.}, 4(2):115--142, 2002.

\bibitem[Eli91]{El91}
Y.~Eliashberg.
\newblock {On symplectic manifolds with some contact properties}.
\newblock {\em Journal of Differential Geometry}, 33(1):233 -- 238, 1991.

\bibitem[Eva10]{Ev10}
L.~Evans.
\newblock {\em Partial differential equations}.
\newblock American Mathematical Society, Providence, R.I., 2010.

\bibitem[Fau16a]{Fa16}
A.~Fauck.
\newblock Handle attaching in symplectic topology - a second glance.
\newblock https://www.mathematik.hu-berlin.de/~wendl/SFT8/Fauck.pdf, 2016.

\bibitem[Fau16b]{Fa16-phd}
A.~Fauck.
\newblock {\em {Rabinowitz-Floer homology on Brieskorn manifolds}}.
\newblock PhD thesis, {Humboldt-Universit\"at zu Berlin}, 2016.

\bibitem[Fau20]{Fa20}
A.~Fauck.
\newblock On manifolds with infinitely many fillable contact structures.
\newblock {\em Int. J. Math.}, 31(13), 2020.

\bibitem[Flo88]{Fl88}
A.~Floer.
\newblock Morse theory for {L}agrangian intersections.
\newblock {\em J. Diff. Geom.}, 28(3):513--547, 1988.

\bibitem[FVK18]{FVK18}
U.~Frauenfelder and O.~Van~Koert.
\newblock {\em {The restricted Three-Body problem and holomorphic curves}}.
\newblock Springer, 2018.

\bibitem[Gei08]{Ge08}
H.~Geiges.
\newblock {\em An Introduction to Contact Topology}.
\newblock Cambridge Studies in Advanced Mathematics. Cambridge University
  Press, 2008.

\bibitem[Gro85]{Gr85}
M.~Gromov.
\newblock Pseudo holomorphic curves in symplectic manifolds.
\newblock {\em Invent. Math.}, 82:307--347, 1985.

\bibitem[HT11]{HT11}
M.~Hutchings and C.~H. Taubes.
\newblock Proof of the {Arnold} chord conjecture in three dimensions. {I}.
\newblock {\em Math. Res. Lett.}, 18(2):295--313, 2011.

\bibitem[HT13]{HT13}
M.~Hutchings and C.~H. Taubes.
\newblock Proof of the {Arnold} chord conjecture in three dimensions. {II}.
\newblock {\em Geom. Topol.}, 17(5):2601--2688, 2013.

\bibitem[Iri13]{Ir13}
K.~Irie.
\newblock Handle attaching in wrapped {F}loer homology and brake orbits in
  classical {H}amiltonian systems.
\newblock {\em Osaka J. Math.}, 50(2):363--396, 2013.

\bibitem[McL08]{Mc08}
M.~McLean.
\newblock {\em {The Symplectic Topology of Stein Manifolds}}.
\newblock PhD thesis, University of Cambridge, 2008.

\bibitem[McL09]{Mc09}
M.~McLean.
\newblock Lefschetz fibrations and symplectic homology.
\newblock {\em {Geometry \& Topology}}, 13(4):1877–1944, 2009.

\bibitem[Moh01]{Mo01}
K.~Mohnke.
\newblock Holomorphic disks and the chord conjecture.
\newblock {\em Ann. Math.}, 154:219--222, 2001.

\bibitem[MS12]{MS12}
D.~McDuff and D.~Salamon.
\newblock {\em $J$-holomorphic curves and Symplectic Topology}.
\newblock American Mathematical Society, Colloquium Publications, 2nd edition,
  2012.

\bibitem[MS17]{MS17}
D.~McDuff and D.~Salamon.
\newblock {\em Introduction to Symplectic Topology}.
\newblock Oxford University Press, 3rd edition, 2017.

\bibitem[MU19]{MU19}
W.~Merry and I.~Uljarevic.
\newblock Maximum principles in symplectic homology.
\newblock {\em Israel Journal of Mathematics}, 229(1):39--65, 2019.

\bibitem[Oan04]{Oa04}
A.~Oancea.
\newblock {A survey of Floer homology for manifolds with contact type boundary
  or Symplectic homology}.
\newblock arXiv:math/0403377, 2004.

\bibitem[PR14]{PR14}
L.~Polterovich and D.~Rosen.
\newblock {\em {Function Theory on Symplectic Manifolds}}.
\newblock CRM monograph series. American Mathematical Society, 2014.

\bibitem[Rit13]{Ri13}
A.~Ritter.
\newblock Topological quantum field theory structure on symplectic cohomology.
\newblock {\em J. Topol.}, 6:391--489, 2013.

\bibitem[Rot09]{Ro09}
J.~Rotman.
\newblock {\em {An Introduction to Homological Algebra}}.
\newblock Springer New York, 2 edition, 2009.

\bibitem[RS93]{RS93}
J.~Robbin and D.~Salamon.
\newblock The {M}aslov index for paths.
\newblock {\em Topology}, 32(4):827--844, 1993.

\bibitem[Sal97]{Sa97}
D.~Salamon.
\newblock Lectures on {F}loer homology.
\newblock IAS/Park City Graduate Summer School on Symplectic Geometry and
  Topology, 1997.

\bibitem[Sei00]{Se00}
P.~Seidel.
\newblock Graded lagrangian submanifolds.
\newblock {\em Bull. Soc. Math. France}, page 103–146, 2000.

\bibitem[Sei07]{Se07}
P.~Seidel.
\newblock {A biased view of symplectic cohomology}.
\newblock arXiv:0704.2055, 2007.

\bibitem[SS05]{SS05}
P.~Seidel and I.~Smith.
\newblock {The symplectic topology of Ramanujam’s surface}.
\newblock {\em Comment. Math. Helv.}, 80:859–881, 2005.

\bibitem[SW06]{SW06}
D.~Salamon and J.~Weber.
\newblock Floer homology and the heat flow.
\newblock {\em Geom. Funct. Anal.}, 16:1050--1138, 2006.

\bibitem[Vit99]{Vi99}
C.~Viterbo.
\newblock Functors and computations in {F}loer homology with applications, {I}.
\newblock {\em Geom. Funct. Anal.}, 9(5):985--1033, 1999.

\bibitem[Wei91]{We91}
A.~Weinstein.
\newblock {Contact surgery and symplectic handlebodies}.
\newblock {\em Hokkaido Mathematical Journal}, 20(2):241 -- 251, 1991.

\bibitem[Wen]{We}
C.~Wend.
\newblock {A BEGINNER’S OVERVIEW OF SYMPLECTIC HOMOLOGY}.
\newblock https://www.mathematik.hu-berlin.de/~wendl/pub/SH.pdf.

\end{thebibliography}
	\bibliographystyle{alpha}
}
\Address
\end{document}